\documentclass{amsproc}

\usepackage{amsmath, amssymb, amsthm, amscd}

\newtheorem{theorem}{Theorem}[section]
\newtheorem{corollary}[theorem]{Corollary}
\newtheorem{lemma}[theorem]{Lemma}
\newtheorem{proposition}[theorem]{Proposition}

\theoremstyle{definition}

\newtheorem{remarks}[theorem]{Remarks}

\theoremstyle{remark}
\newtheorem{remark}[theorem]{Remark}

\numberwithin{equation}{section}

\usepackage[all]{xy}

\renewcommand\-{\text -}

\newcommand\bA{{\mathbb A}}

\newcommand\bD{{\mathbb D}}
\newcommand\bF{{\mathbb F}}
\newcommand\bG{{\mathbb G}}

\newcommand\bP{{\mathbb P}}
\newcommand\bQ{{\mathbb Q}}

\newcommand\bZ{{\mathbb Z}}

\newcommand\bfA{{\bf A}}
\newcommand\bfB{{\bf B}}
\newcommand\bfC{{\bf C}}
\newcommand\bfD{{\bf D}}
\newcommand\bfE{{\bf E}}
\newcommand\bfF{{\bf F}}
\newcommand\bfG{{\bf G}}
\newcommand\bfM{{\bf M}}
\newcommand\bfL{{\bf L}}
\newcommand\bfR{{\bf R}}
\newcommand\bfT{{\bf T}}
\newcommand\bfU{{\bf U}}
\newcommand\bfX{{\bf X}}

\newcommand\cA{{\mathcal A}}

\newcommand\cC{{\mathcal C}}
\newcommand\cD{{\mathcal D}}
\newcommand\cE{{\mathcal E}}
\newcommand\cF{{\mathcal F}}

\newcommand\cL{{\mathcal L}}
\newcommand\cM{{\mathcal M}}

\newcommand\cO{{\mathcal O}}
\newcommand\cP{{\mathcal P}}

\newcommand\cS{{\mathcal S}}
\newcommand\cT{{\mathcal T}}
\newcommand\cU{{\mathcal U}}
\newcommand\cV{{\mathcal V}}
\newcommand\cX{{\mathcal X}}
\newcommand\cY{{\mathcal Y}}

\newcommand\tX{{\widetilde X}}

\newcommand\tbfL{\widetilde{\bf L}}
\newcommand\tbfM{\widetilde{\bf M}}
\newcommand\tbfR{\widetilde{\bf R}}
\newcommand\tbfT{\widetilde{\bf T}}

\newcommand\tcC{\widetilde{\mathcal{C}}}
\newcommand\tcF{\widetilde{\mathcal{F}}}
\newcommand\tcL{\widetilde{\mathcal{L}}}
\newcommand\tcM{\widetilde{\mathcal{M}}}

\newcommand\tcU{\widetilde{\mathcal{U}}}
\newcommand\tcX{\widetilde{\mathcal{X}}}

\newcommand\tucC{\widetilde{\underline{\mathcal{C}}}}
\newcommand\tucL{\widetilde{\underline{\mathcal{L}}}}
\newcommand\tucS{\widetilde{\underline{\mathcal{S}}}}
\newcommand\tucT{\widetilde{\underline{\mathcal{T}}}}
\newcommand\tucU{\widetilde{\underline{\mathcal{U}}}}

\newcommand\ucA{\underline{\mathcal A}}
\newcommand\ucC{\underline{\mathcal C}}

\newcommand\ucL{\underline{\mathcal L}}
\newcommand\ucS{\underline{\mathcal S}}
\newcommand\ucT{\underline{\mathcal T}}
\newcommand\ucU{\underline{\mathcal U}}
\newcommand\ucV{\underline{\mathcal V}}

\newcommand\wA{\widehat{A}}

\newcommand\wC{\widehat{C}}

\newcommand\wcC{\widehat{\mathcal{C}}}

\renewcommand\char{{\rm char}}

\newcommand\hd{{\rm hd}}
\newcommand\id{{\rm id}}

\renewcommand\mod{{\rm mod}}
\newcommand\op{{\rm op}}

\newcommand\tors{{\rm tors}}

\newcommand\D{{\rm D}}

\newcommand\R{{\rm R}}
\newcommand\U{{\rm U}}

\newcommand\Aut{{\rm Aut}}

\newcommand\Coker{{\rm Coker}}

\newcommand\End{{\rm End}}
\newcommand\Ext{{\rm Ext}}

\newcommand\Fract{{\rm Fract}}
\newcommand\Gal{{\rm Gal}}
\newcommand\GL{{\rm GL}}

\newcommand\Hom{{\rm Hom}}
\renewcommand\Im{{\rm Im}}

\newcommand\Ker{{\rm Ker}}

\renewcommand\mod{{\rm mod}}
\newcommand\Mod{{\rm Mod}}
\newcommand\Modfg{{\rm Mod}^{\rm fg}}

\newcommand\Modss{{\rm Mod}^{\rm ss}}
\newcommand\Modtors{{\rm Mod}^{\rm tors}}

\newcommand\SL{{\rm SL}}
\newcommand\Spec{{\rm Spec}}

\newcommand\tmod{\widetilde{{\rm mod}}}

\begin{document}

\title{Commutative algebraic groups up to isogeny. II}

\author{Michel Brion}

\address{Universit\'e Grenoble Alpes, Institut Fourier, CS 40700,
38058 Grenoble cedex 09, France}

\email{Michel.Brion@univ-grenoble-alpes.fr}



\begin{abstract}
This paper develops a representation-theoretic approach to the isogeny 
category $\ucC$ of commutative group schemes of finite type over a field $k$, 
studied in \cite{Brion-II}. We construct a ring $R$ such that 
$\ucC$ is equivalent to the category $R$-mod of all left $R$-modules 
of finite length. We also construct an abelian category of $R$-modules, 
$R$-$\tmod$, which is hereditary, has enough projectives, 
and contains $R$-mod as a Serre subcategory; this yields a more 
conceptual proof of the main result of \cite{Brion-II}, asserting that 
$\ucC$ is hereditary. We show that $R$-$\tmod$ is equivalent to the 
isogeny category of commutative quasi-compact $k$-group schemes. 
\end{abstract}

\maketitle

\tableofcontents

\section{Introduction}
\label{sec:i}
 
In this paper, we develop a representation-theoretic approach to 
the isogeny category of commutative algebraic groups over a field $k$, 
studied in \cite{Brion-II}. This abelian category, that we 
denote by $\ucC$, is equivalent to the quotient of the abelian category 
$\cC$ of group schemes of finite type over $k$ by the Serre subcategory 
$\cF$ of finite $k$-group schemes. The main result of \cite{Brion-II} 
asserts that $\ucC$ is \emph{hereditary}, i.e., $\Ext^i_{\ucC}(G,H) = 0$ 
for all $i \geq 2$ and all $G,H \in \ucC$. 
By a theorem of Serre (see \cite[10.1]{Serre-GP}) which was 
the starting point of our work, this also holds for the original 
category $\cC$ when $k$ is algebraically closed of characteristic 
$0$. But for certain fields $k$, the extension groups in $\cC$ 
can be non-zero in arbitrarily large degrees, as a consequence of 
\cite[Thm.~1]{Milne}.

To prove that $\ucC$ is hereditary, the approach of \cite{Brion-II} 
is similar to that of Serre in \cite{Serre-GP}, generalized by 
Oort in \cite{Oort66} to determine the extension groups in $\cC$ when 
$k$ is algebraically closed of positive characteristic. As $\ucC$ is 
easily seen to be a finite length category, it suffices to check the 
vanishing of higher extension groups for all simple objects 
$G,H \in \ucC$. These are the additive group, the simple tori and 
the simple abelian varieties, and one may then adapt the case-by-case
analysis of \cite{Serre-GP, Oort66} to the easier setting  of the
isogeny category.

In this paper, we obtain a more conceptual proof, by constructing a ring 
$R$ such that $\ucC$ is equivalent to the category $R$-mod of all left 
$R$-modules of finite length; moreover, $R$-mod is a Serre category
of an abelian category $R$-$\tmod$ of left $R$-modules, which is hereditary 
and has enough projectives. For a more precise statement, we refer to Theorem 
\ref{thm:main} in \S \ref{subsubsec:gagc}, 
which can be read independently of the rest of the paper. Our result 
generalizes, and builds on, the equivalence of the isogeny category 
of unipotent groups over a perfect field of positive characteristic with the 
category of modules of finite length over a localization of the Dieudonn\'e 
ring (see \cite[V.3.6.7]{DG}).

More specifically, the abelian category $\cC$ has very few projectives: 
the unipotent groups in characteristic $0$, and the trivial group 
otherwise (see \cite[Thm.~2.9, Cor.~5.15]{Brion-II})).
This drawback was remedied in \cite{Serre-GP} by considering the abelian 
category $\wC$ of \emph{pro-algebraic groups}. If $k$ is algebraically 
closed of characteristic $0$, then $\wcC$ is hereditary, has enough
projectives, and contains $\cC$ as a Serre subcategory. 
We obtain a similar result for the isogeny category $\ucC$ 
over an arbitrary field; this category has more projectives than 
$\cC$ (e.g., the tori), but still not enough of them. We show that $\ucC$ 
is a Serre subcategory of the isogeny category $\tucC$ of 
\emph{quasi-compact group schemes}, which is a hereditary 
abelian category having enough projectives. In addition, $\tucC$ 
is equivalent to the category $R$-$\tmod$ mentioned above
(see again \S \ref{subsubsec:gagc} for a more precise statement).
The quasi-compact group schemes, studied by Perrin in 
\cite{Perrin, Perrin-II}, form a restricted class of pro-algebraic 
groups, discussed in more details in \S \ref{subsubsec:gagqc}.

Since $\ucC$ is a length category, it is equivalent to the category 
of all left modules of finite length over a basic pseudo-compact ring 
$A$, which is then uniquely determined (this result is due to 
Gabriel, see \cite[IV.4]{Gabriel} and \cite[7.2]{Gabriel-II}).
The ring $R$ that we construct is also basic, but not pseudo-compact;
it may be viewed as a dense subring of $A$. Its main advantage 
for our purposes is that the above category $R$-$\tmod$ consists of 
$R$-modules but not of $A$-modules.

This paper is organized as follows. In Section \ref{sec:chc}, 
we study homological properties of abelian categories equipped with 
a torsion pair. This setting turns out to be very useful when 
dealing with algebraic groups, since these are obtained as extensions 
of groups of special types: for example, every connected algebraic 
group is an extension of an abelian variety by a affine algebraic 
group, and these are unique up to isogeny. The main result of 
Section \ref{sec:chc} is Theorem \ref{thm:equiv}, which explicitly 
describes certain abelian categories equipped with a torsion pair,
in terms of modules over triangular matrix rings.

Section \ref{sec:acag} begins with a brief survey of the structure
theory for commutative algebraic groups, with emphasis on categorical
aspects. We also treat in parallel the affine group schemes
(which form the pro-completion of the abelian category of affine
algebraic groups) and the quasi-compact group schemes. We then
obtain our main Theorem \ref{thm:main} by combining all these
structure results with Theorem \ref{thm:equiv}. Next, after some
auxiliary developments in Subsection \ref{subsec:fp}, we study the
finiteness properties of the spaces of morphisms and extensions
in the isogeny category $\ucC$. In particular, we show that
$\ucC$ is $\bQ$-linear, Hom- and Ext-finite if and only if $k$
is a number field (Proposition \ref{prop:veavf}). The final
Subsection \ref{subsec:frt} initiates the study of the indecomposable
objects of $\ucC$, by considering a very special situation: extensions 
of abelian varieties with prescribed simple factors by unipotent groups,
over the field of rational numbers. Using a classical result of Dlab and 
Ringel on representations of species (see \cite{DR}), we obtain 
a characterization of finite representation type in that setting 
(Proposition \ref{prop:frt}).

\medskip

\noindent
{\bf Notation and conventions.}
All considered categories are assumed to be small. We denote categories
by calligraphic letters, e.g., $\cX, \cY$, and functors by boldface letters,
e.g., $\bfL, \bfR$. By abuse of notation, we write $X \in \cX$ if $X$ is 
an object of $\cX$. Also, we say that $\cX$ contains $\cY$ if $\cY$ is 
a full subcategory of $\cX$.

For any ring $R$, we denote by $R$-$\Mod$ the category of left $R$-modules, 
and by $R$-$\Modfg$ (resp.~$R$-$\Modss$, $R$-$\mod$) the full subcategory 
of finitely generated modules (resp.~of semi-simple modules, of modules of 
finite length).

\section{A construction of hereditary categories}
\label{sec:chc}

\subsection{Two preliminary results}
\label{subsec:prel}

Let $\cC$ be an abelian category. Recall that the 
\emph{homological dimension} of $\cC$  is the smallest non-negative 
integer $n =: \hd(\cC)$  such that $\Ext^{n + 1}_{\cC}(X,Y) = 0$ 
for all $X,Y \in \cC$; equivalently, $\Ext^m_{\cC}(X,Y) = 0$ for
all $X,Y \in \cC$ and all $m > n$. If there is no such integer, then
$\hd(\cC)$ is understood to be infinite. 

Also, recall that $\cC$ is said to be \emph{semi-simple} if $\hd(\cC) = 0$; 
equivalently, every short exact sequence in $\cC$ splits. 
If $\hd(\cC) \leq 1$, then $\cC$ is said to be \emph{hereditary}.

We now record two easy lemmas, for which we could not locate 
appropriate references.

\begin{lemma}\label{lem:right}
The following conditions are equivalent for an abelian category
$\cC$ and a non-negative integer $n$:

\begin{enumerate}

\item[{\rm (i)}] $\hd(\cC) \leq n$.

\item[{\rm (ii)}] The functor $\Ext^n_{\cC}(X,?)$ is right exact 
for any $X \in \cC$.

\item[{\rm (iii)}] The functor $\Ext^n_{\cC}(?,Y)$ is right exact 
for any $Y \in \cC$.

\end{enumerate}

\end{lemma}

\begin{proof}
(i) $\Rightarrow$ (ii), (i) $\Rightarrow$ (iii) 
Both assertions follow from the vanishing of $\Ext^{n+1}_{\cC}(X,?)$ 
in view of the long exact sequence of Ext groups.

(ii) $\Rightarrow$ (i) 
Let $\xi \in \Ext^{n+1}_{\cC}(X,Y)$ be the class of an exact sequence
\[ 0 \longrightarrow Y \longrightarrow X_{n+1} \longrightarrow \cdots
\longrightarrow X_1 \longrightarrow X \longrightarrow 0 \]
in $\cC$. We cut this sequence in two short exact sequences
\[ 0 \longrightarrow Y \longrightarrow X_{n+1} \longrightarrow Z
\longrightarrow 0,  
\quad 0 \longrightarrow Z \longrightarrow X_n \longrightarrow 
\cdots \longrightarrow X_1 \longrightarrow X \longrightarrow 0 \]
with classes $\xi_1 \in \Ext^1_{\cC}(Z,Y)$,
$\xi_2 \in \Ext^n_{\cC}(X,Z)$ respectively. Then $\xi$ is 
the Yoneda product $\xi_1 \cdot \xi_2$. Since the natural map 
$\Ext^n_{\cC}(X,X_{n+1}) \to \Ext^n_{\cC}(X,Z)$ is surjective, 
there exists a commutative diagram with exact rows
\[ \CD
0 @>>> X_{n+1} @>>> X'_n @>>> \cdots @>>> X'_1 @>>> X @>>> 0 \\ 
& & @VVV @VVV & & @VVV @V{\id}VV \\
0 @>>> Z @>>> X_n @>>> \cdots @>>> X_1 @>>> X @>>> 0. \\
\endCD \]
Also, we have a commutative diagram with exact rows
\[ \CD
0 @>>> Y @>>> X'_{n+1} @>{f'}>> X_{n+1} @>>> 0 \\ 
& & @V{\id}VV @VVV @V{f}VV \\
0 @>>> Y @>>> X_{n+1} @>{f}>> Z @>>> 0, \\ 
\endCD \]
where the top exact sequence is split (as $\id_{X_{n+1}}$ yields a
section of $f'$). Thus, $f^*(\xi_1) = 0$ in $\Ext^1_{\cC}(X_{n+1},Y)$.
By concatenating both diagrams, we obtain a morphism of extensions
\[ \CD
0 @>>> Y @>>> X'_{n+1} @>>> \cdots @>>> X'_1 @>>> X @>>> 0 \\ 
& & @V{\id}VV @VVV & & @VVV @V{\id}VV \\
0 @>>> Y @>>> X_{n+1} @>>> \cdots @>>> X_1 @>>> X @>>> 0. \\
\endCD \]
Thus, $\xi$ is also represented by the top exact sequence, and hence
$\xi = f^*(\xi_1) \cdot \xi_2 = 0$. This completes the proof of
(ii) $\Rightarrow$ (i).

A dual argument shows that (iii) $\Rightarrow$ (i). 
\end{proof}

Next, recall that a subcategory $\cD$ of an abelian
category $\cC$ is said to be a \emph{Serre subcategory}
if $\cD$ is non-empty, strictly full in $\cC$, and stable 
under taking subobjects, quotients and extensions.

\begin{lemma}\label{lem:hereditary}
Let $\cC$ be a hereditary abelian category, and $\cD$ 
a Serre subcategory. Then $\cD$ is hereditary.
\end{lemma}

\begin{proof}
In view of the assumption on $\cD$, the natural map 
$\Ext^1_{\cD}(X,Y) \to \Ext^1_{\cC}(X,Y)$ 
is an isomorphism for any $X,Y \in \cD$. So the assertion
follows from Lemma \ref{lem:right}.
\end{proof}

\subsection{Torsion pairs}
\label{subsec:tp}

Throughout this subsection, we consider an abelian category $\cC$
equipped with a \emph{torsion pair}, that is, a pair of strictly full 
subcategories $\cX$, $\cY$ satisfying the following conditions:

\begin{enumerate}

\item[{\rm (i)}] $\Hom_{\cC}(X,Y)= 0$ for all $X \in \cX$, $Y \in \cY$.

\item[{\rm (ii)}] For any $C \in \cC$, there exists an exact sequence
in $\cC$
\begin{equation}\label{eqn:torsion}  
0 \longrightarrow X_C 
\stackrel{f_C}{\longrightarrow} C 
\stackrel{g_C}{\longrightarrow} Y_C 
\longrightarrow 0, 
\end{equation}
where $X_C \in \cX$ and $Y_C \in \cY$.

\end{enumerate}

Then $\cX$ is stable under quotients, extensions and coproducts, and
$\cY$ is stable under subobjects, extensions and products. Moreover,
the assignment $C \mapsto X_C$ extends to an additive functor 
$\bfR : \cC \to \cX$, right adjoint to the inclusion. Dually, the assignment
$C \mapsto Y_C$ extends to an additive functor $\bfL : \cC \to \cY$, 
left adjoint to the inclusion (see e.g.~\cite[Sec.~1.1]{BR} for these
results).

\begin{lemma}\label{lem:exact}
Assume that $\cX$, $\cY$ are Serre subcategories of $\cC$. Then:

\begin{enumerate}

\item[{\rm (i)}] $\bfR$, $\bfL$ are exact. 

\item[{\rm (ii)}] $\Ext^n_{\cC}(X,Y) = 0$ for all $X \in \cX$, 
$Y \in \cY$ and $n \geq 1$.

\item[{\rm (iii)}] $\hd(\cC) \leq \hd(\cX) + \hd(\cY) + 1$.

\end{enumerate}

\end{lemma}

\begin{proof}
(i) Let $C \in \cC$ and consider a subobject $i : C_1 \hookrightarrow C$. 
Denote by $q: C \to C_2 := C/C_1$ the quotient map, and by
$C_1 \cap X_C$ the kernel of the map 
$(q,g_C) : C \to C_2 \times Y_C$. 
Then $C_1 \cap X_C \hookrightarrow X_C$,
and hence $C_1 \cap  X_C \in \cC$. Moreover, 
$C_1/C_1 \cap X_C \hookrightarrow C/X_C \cong Y_C$,
and hence $C_1/C_1 \cap X_C \in \cY$. Thus, 
$C_1 \cap X_C = \bfR(C_1)$ and $C_1/C_1 \cap X_C = \bfL(C_1)$.
So we obtain a commutative diagram of exact sequences
\[ \CD
0 @>>> \bfR(C_1) @>{\bfR(i)}>> \bfR(C) @>>> \bfR(C)/\bfR(C_1) @>>> 0 \\ 
& & @V{f_{C_1}}VV @V{f_C}VV @V{g}VV \\
0 @>>> C_1 @>{i}>> C @>{q}>> C_2 @>>> 0. \\ 
\endCD \]
As we just showed, the left square is cartesian; it follows that $g$
is a monomorphism. This yields a commutative diagram
with exact rows and columns
\[ \CD
& & 0 & & 0 & & 0 \\
& & @VVV @VVV @VVV \\ 
0 @>>> \bfR(C_1) @>{\bfR(i)}>> \bfR(C) @>>> \bfR(C)/\bfR(C_1) @>>> 0 \\ 
& & @V{f_{C_1}}VV @V{f_C}VV @V{g}VV \\
0 @>>> C_1 @>{i}>> C @>{q}>> C_2 @>>> 0 \\
& & @VVV @VVV @VVV \\
0 @>>> \bfL(C_1) @>>> \bfL(C) @>>> C_2/\Im(g) @>>> 0 \\
& & @VVV @VVV @VVV \\
& & 0 & & 0 & & 0 
\endCD \]
Moreover, $\bfR(C)/\bfR(C_1) \in \cX$ and $C_2/\Im(g) \in \cY$
by our assumption on $\cX$, $\cY$. It follows that 
$\bfR(C)/\bfR(C_1) = \bfR(C_2)$ and $C_2/\Im(g) = \bfL(C_2)$. 
Thus, $\bfR, \bfL$ are exact.

(ii) We first show that $\Ext^1_{\cC}(X,Y) = 0$. 
Consider an exact sequence
\[ 0 \longrightarrow Y \longrightarrow C \longrightarrow X 
\longrightarrow 0 \] 
in $\cC$. Then the induced map
$\bfR(C) \to \bfR(X) = X$ is an isomorphism, since $\bfR$ is exact
and $\bfR(Y) = 0$. Thus, the above exact sequence splits; this 
yields the assertion.

Next, we show the vanishing of any $\xi \in \Ext^n_{\cC}(X,Y)$.
For this, we adapt the argument of Lemma \ref{lem:right}. Choose
a representative of $\xi$ by an exact sequence
\[ 0 \longrightarrow Y \longrightarrow C_n \longrightarrow \cdots
\longrightarrow C_1 \longrightarrow X \longrightarrow 0 \]
in $\cC$ and cut it in two short exact sequences
\[ 0 \longrightarrow Y \longrightarrow C_n \longrightarrow Z
\longrightarrow 0,  
\quad 0 \longrightarrow Z \longrightarrow C_{n-1} \longrightarrow 
\cdots \longrightarrow C_1 \longrightarrow X \longrightarrow 0. \]
This yields an exact sequence
\[ 0 \longrightarrow \bfR(Z) \longrightarrow \bfR(C_{n-1}) 
\longrightarrow \cdots \longrightarrow \bfR(C_1) \longrightarrow 
X \longrightarrow 0. \]
Also, we obtain a commutative diagram of exact sequences
\[ \CD
0 @>>> Y @>>> C'_n @>>> \bfR(Z) @>>> 0 \\ 
& & @V{\id}VV @VVV @VVV \\
0 @>>> Y @>>> C_n @>>> Z @>>> 0, \\ 
\endCD \]
where the top sequence splits by the above step. So $\xi$ is also
represented by the exact sequence
\[ 0 \longrightarrow Y \longrightarrow C'_n \longrightarrow 
\bfR(C_{n-1}) \longrightarrow \cdots
\longrightarrow \bfR(C_1) \longrightarrow X \longrightarrow 0, \]
which has a trivial class in $\Ext^n_{\cC}(X,Y)$.

(iii) We may assume that $\hd(\cX) := m$ and $\hd(\cY) := n$ 
are both finite. In view of the exact sequence (\ref{eqn:torsion}) 
and the long exact sequence for Ext groups, it suffices to show 
that $\Ext^{m + n + 2}_{\cC}(C,C') = 0$ for all $C,C'$ in $\cX$ or $\cY$.
By (ii), this holds whenever $C \in \cX$ and $C' \in \cY$.
Also, if $C,C' \in \cX$, then $\Ext^{m + 1}_{\cC}(C,C') = 0$:
indeed, every exact sequence
\[ 0 \longrightarrow C' \longrightarrow C_{m + 1} \longrightarrow
\cdots \longrightarrow C_1 \longrightarrow C \longrightarrow 0 \]
is Yoneda equivalent to the exact sequence
\[ 0 \longrightarrow C' \longrightarrow \bfR(C_{m + 1}) \longrightarrow
\cdots \longrightarrow \bfR(C_1) \longrightarrow C \longrightarrow 0, \]
which in turn is equivalent to $0$ by assumption.
Likewise, $\Ext^{n+1}_{\cC}(C,C') = 0$ for all $C, C' \in \cY$. 
So we are reduced to checking that $\Ext^{m + n + 2}_{\cC}(Y,X) = 0$
for all $X \in \cX$, $Y \in \cY$.

For this, we adapt again the argument of Lemma \ref{lem:right}.
Let $\xi \in \Ext^{m + n + 2}_{\cC}(Y,X)$ be represented by an exact
sequence 
\[ 0 \longrightarrow X \longrightarrow C_{m + n + 2} 
\longrightarrow \cdots \longrightarrow C_1
\longrightarrow Y \longrightarrow 0 \]
in $\cC$. This yields two exact sequences
\[ 0 \longrightarrow X \longrightarrow C_{m + n + 2} 
\longrightarrow \cdots \longrightarrow C_{n +2}
\longrightarrow Z \longrightarrow 0,\]
\[ 0 \longrightarrow Z \longrightarrow C_{n + 1} 
\longrightarrow \cdots \longrightarrow C_1
\longrightarrow Y \longrightarrow 0. \]
As $\Ext^{n+1}_{\cC}(Y,\bfL(Z)) = 0$ by the above step,
the exact sequence
\[ 0 \longrightarrow \bfR(Z) \longrightarrow Z 
\longrightarrow \bfL(Z) \longrightarrow 0 \]
yields a surjection 
$\Ext^{n+1}_{\cC}(Y,\bfR(Z)) \to \Ext^{n+1}_{\cC}(Y,Z)$.
Thus, there exists a commutative diagram of exact sequences
\[ \CD
0 @>>> \bfR(Z) @>>> C'_{n+1} @>>> \cdots @>>> C'_1 @>>> Y @>>> 0 \\ 
& & @VVV @VVV & & @VVV @V{\id}VV \\
0 @>>> Z @>>> C_{n+1} @>>> \cdots @>>> C_1 @>>> Y @>>> 0. \\
\endCD \]
Also, we have an exact sequence 
\[ 0 \longrightarrow X \longrightarrow \bfR(C_{m + n + 2}) 
\longrightarrow \cdots \longrightarrow \bfR(C_{n+2})
\longrightarrow \bfR(Z) \longrightarrow 0, \]
with trivial class as $\Ext^{m+1}_{\cC}(\bfR(Z),X) = 0$. Hence $\xi$
is also represented by the exact sequence
\[ 0 \to X \to \bfR(C_{m + n + 2}) \to \cdots \to \bfR(C_{n+2})
\to C'_{n+1} \to \cdots \to C'_1 \to Y \to 0, \]
which has a trivial class as well.
\end{proof}

\begin{corollary}\label{cor:hereditary}
Assume that $\cX$ and $\cY$ are semi-simple Serre subcategories of $\cC$. 
Then:

\begin{enumerate}

\item[{\rm (i)}] Every object of $\cX$ is projective in $\cC$.

\item[{\rm (ii)}] Every object of $\cY$ is injective in $\cC$.

\item[{\rm (iii)}] $\cC$ is hereditary. 

\end{enumerate}

\end{corollary}

\begin{proof}
(i) Let $X \in \cX$. Then $\Ext^1_{\cC}(X,Y) = 0$ for all $Y \in \cY$,
by Lemma \ref{lem:exact}. Moreover, $\Ext^1_{\cC}(X,X') = 0$ for all 
$X' \in \cX$ by our assumption. In view of the exact sequence 
(\ref{eqn:torsion}), it follows that $\Ext^1_{\cC}(X,C) = 0$ for all 
$C \in \cC$, i.e., $X$ is projective in $\cC$.

(ii) This is checked similarly.

(iii) This follows from Lemma \ref{lem:exact} (iii).

\end{proof}

\subsection{The category of extensions}
\label{subsec:ce}

We still consider an abelian category $\cC$ equipped with
a torsion pair $(\cX,\cY)$.
Let $\cE$ be the category with objects the triples $(X,Y,\xi)$,
where $X \in \cX$, $Y \in \cY$ and $\xi \in \Ext^1_{\cC}(Y,X)$;
the morphisms from $(X,Y,\xi)$ to $(X',Y',\xi')$ are the pairs 
of morphisms $(u : X \to X'$, $v : Y \to Y')$ such that 
$u_*(\xi) = v^*(\xi')$ in $\Ext^1_{\cC}(Y,X')$. We say that
$\cE$ is the \emph{category of extensions} associated with the triple
$(\cC,\cX,\cY)$. 
 
We may assign to any $C \in \cC$, the triple 
\[ \bfT(C) := (\bfR(C),\bfL(C),\xi(C)), \]
where $\xi(C) \in \Ext^1_{\cC}(\bfL(C),\bfR(C))$ denotes the class
of the extension (\ref{eqn:torsion}),
\[  0 \longrightarrow \bfR(C) \longrightarrow C \longrightarrow \bfL(C)
\longrightarrow 0. \]

\begin{lemma}\label{lem:equiv}
Assume that $\cX,\cY$ are Serre subcategories of $\cC$, and 
$\Hom_{\cC}(Y,X) = 0$ for all $X \in \cX$, $Y \in \cY$. Then 
the above assignment extends to a covariant functor $\bfT: \cC \to \cE$, 
which is an equivalence of categories.
\end{lemma}

\begin{proof}
Consider a morphism $f : C \to C'$ in $\cC$. Then $f$ lies in 
a unique commutative diagram of exact sequences
\[ \CD 
0 @>>> X_C @>{f_C}>> C @>{g_C}>> Y_C @>>> 0 \\ 
& & @V{u}VV @V{f}VV @V{v}VV \\
0 @>>> X_{C'} @>>>  C' @>>> Y_{C'} @>>> 0. \\
\endCD \]
Denote by $\xi \in \Ext^1_{\cC}(Y_C,X_C)$, 
$\xi' \in \Ext^1_{\cC}(Y_{C'},X_{C'})$ the classes of the above 
extensions and set $u =: \bfR(f)$, $v =: \bfL(f)$. These fit 
into a commutative diagram with exact rows
\[ \CD 
0 @>>> \bfR(C) @>>> C @>>> \bfL(C) @>>> 0 \\ 
& & @V{\bfR(f)}VV @VVV @V{\id}VV \\
0 @>>> \bfR(C') @>>> D @>>> \bfL(C) @>>> 0 \\
& & @V{\id}VV @VVV @V{\bfL(f)}VV \\
0 @>>> \bfR(C') @>>> C' @>>> \bfL(C') @>>> 0 \\ 
\endCD \]
It follows that $\bfR(f)_*(\xi) = \bfL(f)^*(\xi')$.
Thus, the assignment $f \mapsto (\bfR(f), \bfL(f))$ defines 
the desired covariant functor $\bfT$. We now show that $\bfT$ 
is an equivalence of categories. 

Since $(\cX,\cY)$ is a torsion pair, $\bfT$ is essentially surjective.
We check that it is faithful. Let $C,C' \in \cC$ and consider
$f \in \Hom_{\cC}(C,C')$ such that $\bfR(f) = 0 = \bfL(f)$. 
Then the composition 
$X_C \stackrel{f_C}{\longrightarrow} C 
\stackrel{f}{\longrightarrow} C'$ is zero, and hence $f$ factors
through $g : Y_C \to C'$. Moreover, the composition
$Y_C \stackrel{g}{\longrightarrow} C' 
\stackrel{g_{C'}}{\longrightarrow} Y_{C'}$ is zero, and hence
$g$ factors through $h : Y_C \to X_{C'}$. By our assumption,
$h = 0$; thus, $f = 0$.

Finally, we show that $\bfT$ is full. Let again $C,C' \in \cC$ 
and consider $u : X_C \to X_{C'}$, $v : Y_C \to Y_{C'}$ such that
$u_*(\xi) = v^*(\xi')$, where $\xi$ (resp.~$\xi'$) denotes the class of
the extension (\ref{eqn:torsion}) for $C$ (resp.~$C'$).
Since $\Hom_{\cC}(Y,X) = 0$ for all $X \in \cX$ and $Y \in \cY$, 
these extensions are uniquely determined by their classes, and
in turn by $C$, $C'$. Thus, we have a commutative diagram of
extensions in $\cC$
\[ \CD 
0 @>>> X_C @>{f_C}>> C @>{g_C}>> Y_C @>>> 0 \\ 
& & @V{u}VV @VVV @V{\id}VV \\
0 @>>> X_{C'} @>{i}>> D @>{q}>> Y_C @>>> 0 \\
& & @V{\id}VV @VVV @V{v}VV \\
0 @>>> X_{C'} @>{f_{C'}}>> C' @>{g_{C'}}>> Y_{C'} @>>> 0. \\ 
\endCD \]
This yields a morphism $f : C \to C'$ such that $\bfR(f) = u$
and $\bfL(f) = v$.
\end{proof}

With the assumptions of Lemma \ref{lem:equiv}, the subcategory 
$\cX$ (resp.~$\cY$) of $\cC$ is identified via $\bfT$ with 
the full subcategory of $\cE$ with objects the triples of the 
form $(X,0,0)$ (resp.~$(0,Y,0)$). Assuming in addition that 
\emph{$\cX$ and $\cY$ are semi-simple}, we now obtain 
a description of homomorphism and extension groups in $\cE$:

\begin{proposition}\label{prop:homext}
With the above assumptions, there is an exact sequence
\[ 0 \longrightarrow \Hom_{\cE}(Z,Z') 
\stackrel{\iota}{\longrightarrow}
\Hom_{\cX}(X,X') \times \Hom_{\cY}(Y,Y') \longrightarrow \]
\[ \stackrel{\varphi}{\longrightarrow}
\Ext^1_{\cC}(Y,X') \longrightarrow \Ext^1_{\cE}(Z,Z')
\longrightarrow 0 \]
for any $Z = (X,Y,\xi), Z' = (X',Y',\xi') \in \cE$, 
where $\iota$ denotes the inclusion, and 
$\varphi(u,v) := u_*(\xi') - v^*(\xi)$.
\end{proposition}

\begin{proof}
We have $\Ker(\iota) = 0$ and $\Im(\iota) = \Ker(\varphi)$ 
by the definition of the morphisms in $\cE$. Thus, it suffices 
to check that $\Coker(\varphi) \cong \Ext^1_{\cE}(Z,Z')$.

Consider the exact sequence 
\[ 0 \longrightarrow X' \longrightarrow Z' \longrightarrow Y'
\longrightarrow 0 \]
in $\cE$, with class $\xi' \in \Ext^1_{\cE}(Y',X') = \Ext^1_{\cC}(Y',X')$. 
This yields an exact sequence
\[ \Hom_{\cE}(Y,Y') \stackrel{\partial'}{\longrightarrow}
\Ext^1_{\cE}(Y,X') \longrightarrow
\Ext^1_{\cE}(Y,Z') \longrightarrow \Ext^1_{\cE}(Y,Y'), \]
where $\partial'(v) := v^*(\xi')$ for any 
$v \in \Hom_{\cE}(Y,Y')$. Moreover, since $\cY$ is a semi-simple Serre
subcategory of $\cE$, we have $\Hom_{\cE}(Y,Y') = \Hom_{\cY}(Y,Y')$ and
$\Ext^1_{\cE}(Y,Y') = 0$. So we obtain a natural isomorphism
\begin{equation}\label{eqn:YZ'} 
\Ext^1_{\cE}(Y,Z') \cong \Ext^1_{\cE}(Y,X')/
\{ v^*(\xi') ~\vert~ v \in \Hom_{\cY}(Y,Y') \}. 
\end{equation}
Similarly, the exact sequence
\[ 0 \longrightarrow X \longrightarrow Z \longrightarrow Y
\longrightarrow 0 \]
in $\cE$, with class $\xi \in \Ext^1_{\cC}(Y,X)$,  
yields an exact sequence
\[ \Hom_{\cE}(X,Z') \stackrel{\partial}{\longrightarrow}
\Ext^1_{\cE}(Y,Z') \longrightarrow
\Ext^1_{\cE}(Z,Z') \longrightarrow \Ext^1_{\cE}(X,Z'), \]
where $\partial(u) := u_*(\xi)$ for any 
$u \in \Hom_{\cE}(X,Z')$. Moreover, the natural map
\[ \Hom_{\cX}(X,X') \longrightarrow \Hom_{\cE}(X,Z') \] 
is an isomorphism, since $\Hom_{\cE}(X,Y') = 0$. Also,
$\Ext^1_{\cE}(X,Z') = 0$ by Corollary \ref{cor:hereditary}. 
Hence we obtain a natural isomorphism 
\begin{equation}\label{eqn:ZZ'} 
\Ext^1_{\cE}(Z,Z') \cong \Ext^1_{\cE}(Y,Z')/
\{ u_*(\xi) ~\vert~ u \in \Hom_{\cX}(X,X') \}.
\end{equation}
Putting together the isomorphisms (\ref{eqn:YZ'}) and (\ref{eqn:ZZ'})
yields the desired assertion.
\end{proof}

\subsection{Universal extensions}
\label{subsec:ue}

We still consider an abelian category $\cC$ equipped with
a torsion pair $(\cX,\cY)$, and make the following assumptions:

\begin{enumerate}

\item[(a)] $\cX, \cY$ are Serre subcategories of $\cC$.

\item[(b)] $\cX, \cY$ are semi-simple.

\item[(c)] $\Hom_{\cC}(Y,X) = 0$ for all $X \in \cX$, $Y \in \cY$.

\item[(d)] There exists a covariant exact functor 
$\bfF : \cY \to \tcX$, where $\tcX$ is a semi-simple abelian 
category containing $\cX$ as a Serre subcategory, and a 
bi-functorial isomorphism
\begin{equation}\label{eqn:iso} 
\Ext^1_{\cC}(Y,X) \stackrel{\cong}{\longrightarrow} 
\Hom_{\tcX}(\bfF(Y),X) 
\end{equation}
for all $X \in \cX$, $Y \in \cY$.

\end{enumerate}

\begin{remark}\label{rem:ue}
The above assumptions are satisfied by the isogeny category of 
algebraic groups and some natural Serre subcategories, as we will 
see in Subsection \ref{subsec:mic}. 
Also, assumptions (a), (b) and (c) are just those of 
Corollary \ref{cor:hereditary} and Proposition \ref{prop:homext}. 
Note that a weak version of (d) always holds, where we only require  
$\tcX$ to be a category containing $\cX$: take $\tcX$ 
to be the opposite category of covariant functors from $\cX$ 
to sets; then the functor $\Ext^1_{\cC}(Y,?)$ is an object
of $\tcX$ for any $Y \in \cY$, and the isomorphism
(\ref{eqn:iso}) follows from Yoneda's lemma. But requiring $\tcX$ 
to be abelian and semi-simple is a restrictive assumption.
\end{remark}

Under the four above assumptions, every $C \in \cC$ defines
an extension class $\xi(C) \in \Ext^1_{\cC}(Y_C,X_C)$, and in turn
a morphism $\eta(C) \in \Hom_{\tcX}(\bfF(Y_C),X_C)$. Moreover, 
every morphism $f : C \to C'$ in $\cC$ induces morphisms
$u : X_C \to X_{C'}$, $v : Y_C \to Y_{C'}$ such that the push-forward
$u_* \xi(C) \in \Ext^1_{\cC}(Y,X')$ is identified with
$u \circ \eta(C)  \in \Hom_{\tcX}(\bfF(Y_C),X_{C'})$,
and the pull-back $v^* \xi(C') \in \Ext^1_{\cC}(Y,X')$ is identified with
$\eta(C') \circ F(v)  \in \Hom_{\tcX}(\bfF(Y_C),X_{C'})$. 

It follows that the category of extensions $\cE$ (considered in Subsection
\ref{subsec:ce}) is equivalent to the category $\cF$ with objects 
the triples $(X,Y,\eta)$, where $X \in \cX$, $Y \in \cY$ and 
$\eta \in \Hom_{\tcX}(\bfF(Y),X)$; the morphisms from
$(X,Y,\eta)$ to $(X',Y',\eta')$ are the pairs of
morphisms ($u : X \to X'$, $v: Y \to Y'$) such that
the diagram
\[ \CD
\bfF(Y) @>{F(v)}>> \bfF(Y') \\
@V{\eta}VV @V{\eta'}VV \\
X @>{u}>> X' \\
\endCD \]
commutes. With this notation, Lemma \ref{lem:equiv} yields readily:

\begin{lemma}\label{lem:equiv2}
The assignment $C \mapsto (\bfR(C),\bfL(C),\eta(C))$ extends 
to an equivalence of categories $\cC \stackrel{\cong}{\to} \cF$.
\end{lemma}

Next, consider the category $\tcF$ with objects the triples
$(X,Y,\eta)$, where $X \in \tcX$, $Y \in \cY$ and 
$\eta \in \Hom_{\tcX}(\bfF(Y), X)$; the morphisms are defined
like those of $\cF$. Then one may readily check the following:

\begin{lemma}\label{lem:tilde}
With the above notation, $\cF$ is a Serre subcategory of $\tcF$.
Moreover, the triple $(\tcF,\tcX,\cY)$ satisfies the assumptions 
{\rm (a)}, {\rm (b)}, {\rm (c)}, and {\rm (d)} with the same 
functor $\bfF$.
For any $Z = (X,Y,\eta)$ and $Z' = (X',Y',\eta') \in \tcF$,
we have an exact sequence  
\[ 0 \longrightarrow \Hom_{\tcF}(Z,Z') \stackrel{\iota}{\longrightarrow}
\Hom_{\cX}(X,X') \times \Hom_{\cY}(Y,Y') \longrightarrow \]
\[ \stackrel{\psi}{\longrightarrow}
\Hom_{\tcX}(\tX,X') \longrightarrow \Ext^1_{\tcF}(Z,Z') \longrightarrow 0, \]
where $\tX := \bfF(Y)$ and 
$\psi(u,v) := u \circ \eta - \eta' \circ \bfF(v)$. 
\end{lemma}

We now consider the covariant exact functors like in Lemma 
\ref{lem:exact}:
\[ \tbfR : \tcF \longrightarrow \tcX, \quad 
(X,Y,\eta) \longmapsto X, \quad (u,v) \longmapsto u, \]
\[ \tbfL : \tcF \longrightarrow \cY, \quad
(X,Y,\eta) \longmapsto Y, \quad (u,v) \longmapsto v. \]

\begin{lemma}\label{lem:adjoint} 
With the above notation, the assignment 
$Y \mapsto (\bfF(Y),Y, \id_{\bfF(Y)})$ extends to a covariant
exact functor $\bfE : \cY \to \tcX$, which is left adjoint
to $\tbfL$. 
\end{lemma} 

\begin{proof}
For any morphism $v : Y_1 \to Y_2$ in $\cY$, the induced 
morphism $\bfF(v): \bfF(Y_1) \to \bfF(Y_2)$ satisfies
$(\bfF(v),v) \in \Hom_{\tcF}(\bfE(Y_1),\bfE(Y_2))$
by the definition of the morphisms in $\tcF$. We may thus
set $\bfE(v) := (\bfF(v),v)$. As $\bfF$ is a covariant exact
functor, so is $\bfE$. We now check the adjunction assertion: 
let $Y \in \cY$ and $(X,Y',\eta) \in \tcF$. Then 
$\Hom_{\tcF}(\bfE(Y),(X,Y',\eta))$ consists of the pairs
($u \in \Hom_{\tcX}(\bfF(Y),X)$, $v \in \Hom_{\cY}(Y,Y')$)
such that $u = \eta \circ F(v)$. Thus, the map induced by
$\tbfL$, 
\[ \Hom_{\tcF}(\bfE(Y),(X,Y',\eta)) \longrightarrow
\Hom_{\cY}(Y,Y'), \quad (u,v) \longmapsto v \]
is an isomorphism.
\end{proof}

For any $Y \in \cY$, there is a tautological exact sequence 
\begin{equation}\label{eqn:univ} 
0 \longrightarrow \bfF(Y) \stackrel{\iota}{\longrightarrow}
\bfE(Y) \stackrel{\pi}{\longrightarrow} Y \longrightarrow 0 
\end{equation}
in $\tcF$, which is universal in the following sense:

\begin{proposition}\label{prop:universal}
Let $Z = (X,Y,\eta) \in \tcF$.

\begin{enumerate}

\item[{\rm (i)}] There exists a unique morphism 
$\mu : \bfE(Y) \to Z$ in $\tcF$ such that the 
diagram
\[ \CD
0 @>>> \bfF(Y) @>{\iota}>> \bfE(Y) @>{\pi}>> Y @>>> 0 \\
& & @V{\eta}VV @V{\mu}VV @V{\id_Y}VV \\
0 @>>> X @>{f}>> Z @>{g}>> Y @>>> 0 \\
\endCD \]
commutes, and the left square is cartesian.

\item[{\rm (ii)}] The resulting exact sequence
\begin{equation}\label{eqn:pullback} 
0 \longrightarrow \bfF(Y) 
\stackrel{(\eta,\iota)}{\longrightarrow}
X \times \bfE(Y) \stackrel{f - \mu}{\longrightarrow}
Z \longrightarrow 0 
\end{equation}
is a projective resolution of $Z$ in $\tcF$. In particular,
$\tcF$ has enough projectives.

\item[{\rm (iii)}] $Z$ is projective if and only if 
$\eta$ is a monomorphism; then $Z \cong X' \times \bfE(Y)$ 
for some subobject $X' \hookrightarrow X$.

\item[{\rm (iv)}] $\bfE(Y)$ is a projective cover of $Y$.

\end{enumerate}

\end{proposition}

\begin{proof}
(i) This follows from the isomorphism (\ref{eqn:univ})
together with Yoneda's lemma.

(ii) By Corollary \ref{cor:hereditary}, $X$ and $\bfF(Y)$ 
are projective objects of $\tcF$. Moreover, $\bfE(Y)$ is 
projective as well, since
$\Hom_{\tcF}(\bfE(Y),?) \cong \Hom_{\cY}(Y,\tbfL(?))$,
where $\tbfL$ is exact and $\cY$ is semi-simple.

(iii) If $Z$ is projective, then of course $\Ext^1_{\tcF}(Z,X') = 0$
for all $X' \in \tcX$. In view of the projective resolution
(\ref{eqn:pullback}), it follows that the map
\[ 
\Hom_{\tcX}(X,X') \times \Hom_{\tcX}(\bfE(Y),X') 
\longrightarrow \Hom_{\tcX}(\bfF(Y),X'), \quad 
(u,v) \longmapsto u \circ \eta + v \circ \iota \]
is surjective. As $\Hom_{\tcX}(\bfE(Y),X') = 0$ by Lemma
\ref{lem:adjoint}, this just means that the map
\[ 
\Hom_{\tcX}(X,X') \longrightarrow \Hom_{\tcX}(\bfF(Y),X'), 
\quad (u,v) \longmapsto u \circ \eta \]
is surjective (alternatively, this follows from the exact sequence
of Lemma \ref{lem:tilde}). 
Since $\tcX$ is semi-simple, the pull-back map
\[ \Hom_{\tcX}(\bfF(Y),X') \longrightarrow 
\Hom_{\tcX}(\Ker(\eta),X') \]
is surjective as well. Thus, $\Hom_{\tcX}(\Ker(\eta),X') = 0$
for all $X'$, i.e., $\Ker(\eta) = 0$. Using the semi-simplicity
of $\tcX$ again, we may choose a subobject
$X' \hookrightarrow X$ such that $X = X' \oplus \Im(\eta)$; then
the natural map $X' \times \bfE(Y) \to Z$ is an isomorphism.
Conversely, $X' \times \bfE(Y)$ is projective by (ii).

(iv) This follows from (iii), since we have $\Hom_{\tcF}(X',Y) = 0$
for any $X' \in \tcX$, and 
$\Hom_{\tcF}(\bfE(Y'),Y) \cong \Hom_{\cY}(Y',Y)$ for any $Y' \in \cY$.
\end{proof}

Finally, we obtain two homological characterizations of the universal
objects $\bfE(Y)$, the first one being somewhat analogous to the notion 
of exceptional objects:

\begin{lemma}\label{lem:exc}
The following conditions are equivalent for 
$Z = (X,Y,\eta) \in \tcF$:

\begin{enumerate}

\item[{\rm (i)}] $Z \cong \bfE(Y)$.

\item[{\rm (ii)}] $\End_{\tcF}(Z) \cong \End_{\cY}(Y)$ via $\tbfL$, and
$\Ext^1_{\tcF}(Z,Z) = 0$.

\item[{\rm (iii)}] $\Hom_{\tcF}(Z,X') = 0 = \Ext^1_{\tcF}(Z,X')$ for all
$X' \in \tcX$.

\end{enumerate}

\end{lemma}

\begin{proof}
(i) $\Rightarrow$ (ii), (i) $\Rightarrow$ (iii) This follows from Lemma 
\ref{lem:adjoint} and Proposition \ref{prop:universal}.

(ii) $\Rightarrow$ (i) In view of the exact sequence of 
Lemma \ref{lem:tilde}, we may rephrase the assumption as follows:
for any $v \in \End_{\cY}(Y)$, there exists a unique 
$u \in \End_{\tcX}(X)$ such that $u \circ \eta = \eta \circ F(v)$;
moreover, the map $\psi: (f,g) \mapsto f \circ \eta - \eta \circ F(g)$ 
is surjective. As a consequence, $\eta$ is an epimorphism 
(by the uniqueness of $u$),
and $\Ker(\eta)$ is stable under $F(v)$ for any $v \in \End_{\cY}(Y)$
(by the existence of $u$). Then $\psi(u,v)$ vanishes identically
on $\Ker(\eta)$ for any $u \in \End_{\tcX}(X)$, $v \in \End_{\cY}(Y)$. 
As $\psi$ is surjective, this forces $\Ker(\eta) = 0$. Thus,
$\eta$ is an isomorphism, and hence the pair
$(\eta: \bfF(Y) \to X$, $\id: Y \to Y$) yields an isomorphism
$\bfE(Y) \to Z$ in $\tcF$.

(iii) $\Rightarrow$ (i) In view of the long exact exact sequence of
Lemma \ref{lem:tilde} again, the map 
\[ \psi : \Hom_{\tcX}(X,X') \longrightarrow \Hom_{\tcX}(\bfF(Y),X'),
\quad u \longmapsto u \circ \eta \]
is an isomorphism for any $\eta \in \tcX$. It follows that $\eta$
is an isomorphism as well. 
\end{proof}

\subsection{Relation to module categories}
\label{subsec:rmc}

Let $\cC$ be an abelian category equipped with a torsion pair
$(\cX,\cY)$ satisfying the assumptions (a), (b), (c), (d) 
of Subsection \ref{subsec:ue}. We assume in addition that $\cC$ 
is a \emph{finite length} category, i.e., every object has a
composition series. Then the semi-simple categories $\cX$, $\cY$ 
are of finite length as well; as a consequence, each of them
is equivalent to the category of all left modules of finite length over 
a ring, which can be constructed as follows.

Denote by $I$ the set of isomorphism classes of simple objects 
of $\cX$. Choose a representative $S$ for each class, and let
\[ D_S := \End_{\cX}(S)^{\op}; \]
then $D_S$ is a division ring. Given $X \in \cX$, the group 
$\Hom_{\cX}(S,X)$ is a left $D_S$-vector space of finite dimension; 
moreover, $\Hom_{\cX}(S,X) = 0$ for all but finitely many $S \in I$. 
Thus, \[ \bfM(X) := \bigoplus_{S \in I} \Hom_{\cX}(S,X) \]
is a left module of finite length over the ring
\[ R_{\cX} := \bigoplus_{S \in I} D_S. \]
(Notice that every $R_{\cX}$-module of finite length is semi-simple;
moreover, the ring $R_{\cX}$ is semi-simple if and only if $I$ 
is finite). The assignment $X \mapsto \bfM(X)$ extends to a
covariant functor
\begin{equation}\label{eqn:ss} 
\bfM_{\cX} : \cX \stackrel{\cong}{\longrightarrow} 
R_{\cX}\-\mod, 
\end{equation}
which is easily seen to be an equivalence of categories. 
Likewise, we have an equivalence of categories 
\[ \bfM_{\cY} : \cY \longrightarrow R_{\cY}\-\mod,\] 
where $R_{\cY} := \bigoplus_{T \in J} D_T$. 
 
We now make a further (and final) assumption:

\begin{enumerate}
\item[(e)] The equivalence (\ref{eqn:ss}) extends to an equivalence 
of categories 
\begin{equation}\label{eqn:tss} 
\bfM_{\tcX} : \tcX \stackrel{\cong}{\longrightarrow} 
R_{\cX}\-\Modss. 
\end{equation}
\end{enumerate}

The right-hand side of (\ref{eqn:tss}) is a semi-simple category 
containing $\cX$ as a Serre subcategory, as required by assumption 
(d). When the set $I$ is finite, assumption (e) just means that
$\tcX = R_{\cX}$-$\Mod$. For an arbitrary set $I$, the objects 
of $\tcX$ are the direct sums $X = \bigoplus_{S \in I} X_S$, where
each $X_S$ is a left $D_S$-vector space. We say that $X_S$ is 
the \emph{isotypical component} of $X$ of type $S$; its dimension
(possibly infinite) is the \emph{multiplicity} of $S$ in $X$.

We now turn to the covariant exact functor $\bfF : \cY \to \tcX$.
We may identify each simple object $T$ of $\cY$ with $D_T$, on which
$R_{\cY}$ acts via left multiplication. Then $\bfF(D_T)$ is a 
semi-simple left $R_{\cX}$-module. Also, $\bfF(D_T)$ is 
a right $D_T$-module, via the action of $D_T$ on itself via right 
multiplication, which yields an isomorphism 
\[ D_T \stackrel{\cong}{\longrightarrow} \End_{R_{\cY}}(D_T),\] 
and the ring homomorphism 
$\End_{R_{\cY}}(D_T) \to \End_{R_{\cX}}(\bfF(D_T))$
induced by $\bfF$. As the left and right actions of $D_T$ on itself 
commute, $\bfF(D_T)$ is a $D_T$-$R_{\cY}$-bimodule, which 
we view as a $R_{\cX}$-$R_{\cY}$-bimodule. Let
\[ M := \bigoplus_{T \in J} \bfF(D_T). \]
This is again a $R_{\cX}$-$R_{\cY}$-bimodule, semi-simple
as a $R_{\cX}$-module and as a $R_{\cY}$-module. 
So we may form the triangular matrix ring
\[ R := 
\begin{pmatrix} 
R_{\cY} & 0 \\ 
M & R_{\cX} 
\end{pmatrix} \]
as in \cite[\S III.2]{ARS}. More specifically, $R$ consists
of the triples $(x,y,m)$, where $x \in R_{\cX}$, 
$y \in R_{\cY}$, and $m \in M$; the addition and multiplication
are those of the matrices 
$\left( \begin{smallmatrix} y & 0 \\ m & x \end{smallmatrix} \right)$,
using the bimodule structure of $M$. Note the decomposition
$R = (R_{\cX} \oplus R_{\cY}) \oplus M$,
where $R_{\cX} \oplus R_{\cY}$ is a subring, and $M$ is 
an ideal of square $0$.

We say that a $R$-module $Z$ is \emph{locally finite}, if
$Z = X \oplus Y$ as an $R_{\cX} \oplus R_{\cY}$-module,
where $X$ is a semi-simple $R_{\cX}$-module and $Y$
is an $R_{\cY}$-module of finite length. We denote by 
$R$-$\tmod$ the full subcategory of $R$-$\Mod$ with objects 
the locally finite modules; then $R$-$\mod$ is a Serre 
subcategory of $R$-$\tmod$.
We may now state our main homological result:

\begin{theorem}\label{thm:equiv}
With the above notation and assumptions (a), (b), (c), (d), (e), 
the abelian categories $\cF$ and $\tcF$ are hereditary, and $\tcF$ 
has enough projectives. Moreover, there are compatible equivalences 
of categories
\[ \bfM : \cF \stackrel{\cong}{\longrightarrow} R\-\mod, 
\quad  \tbfM : \tcF \stackrel{\cong}{\longrightarrow} R\-\tmod. \]
\end{theorem}

\begin{proof}
The first assertion follows by combining Corollary 
\ref{cor:hereditary}, Lemma \ref{lem:equiv} and Proposition 
\ref{prop:universal}.

To show the second assertion, we can freely replace $\cX$, 
$\tcX$, $\cY$ with compatibly equivalent categories 
in the construction of $\cF$, $\tcF$. Thus, we may assume that 
$\cX = R_{\cX}$-$\mod$, $\tcX = R_{\cX}$-$\Modss$ and 
$\cY = R_{\cY}$-$\mod$.

The category of all left $R$-modules is equivalent to 
the category of triples $(X,Y,f)$, where $X$ is 
a $R_{\cX}$-module, $Y$ a $R_{\cY}$-module, and 
$f : M \otimes_{R_{\cY}} Y \to X$ a morphism of 
$R_{\cX}$-modules. The morphisms from $(X,Y,f)$ to 
$(X',Y',f')$ are the pairs $(u,v)$, where $u \in \Hom_{\cX}(X,X')$, 
$v \in \Hom_{\cY}(Y,Y')$, and the following diagram commutes:
\[ \CD
M \otimes_{R_{\cY}} Y @>{\id_M \otimes v}>> 
M \otimes_{R_{\cY}} Y' \\
@V{f}VV @V{f'}VV \\
X @>{u}>> X'. \\
\endCD \]
(This result is obtained in \cite[Prop.~III.2.2]{ARS} for modules 
of finite length over an Artin algebra. The proof adapts without 
change to the present setting).
Moreover, the full subcategory $R$-$\mod$ 
(resp.~$R$-$\tmod$) is equivalent to the full subcategory 
of triples $(X,Y,f)$, where $X$ and $Y$ have finite length 
(resp.~$X$ is semi-simple and $Y$ has finite length).

To complete the proof, it suffices to show that the covariant
exact functor $\bfF : \cY \to \tcX$ is isomorphic to 
$M \otimes_{R_{\cY}} -$. As $\bfF$ commutes with finite direct 
sums, we have 
$\bfF = \bigoplus_{T \in J} \bfF_T$ for covariant exact functors
$\bfF_T: D_T$-$\mod \to \tcX$. We now argue as in the proof of
the Eilenberg-Watts theorem (see \cite[Thm.~II.2.3]{Bass}).
Given a left $D_T$-vector space $V$, every $v \in V$ yields
a $D_T$-morphism $v : D_T \to V$, and hence a 
$R_{\cX}$-morphism 
$\bfF_T(v) : \bfF_T(D_T) = \bfF(D_T) \to \bfF_T(V)$.
The resulting map 
$V \to \Hom_{R_{\cX}}(\bfF(D_T), \bfF_T(V))$
is easily checked to be a $D_T$-morphism. In view of the natural
isomorphism
\[ \Hom_{D_T}(V,\Hom_{R_{\cX}}(\bfF(D_T), \bfF_T(V)))
\cong 
\Hom_{R_{\cX}}(\bfF(D_T) \otimes_{D_T} V , \bfF_T(V)), \]
this yields a functorial map 
\[ f_V : \bfF(D_T) \otimes_{D_T} V \longrightarrow \bfF_T(V). \]
When $V = D_T$, one checks that $f_V$ is identified to the identity 
map of $\bfF(D_T)$; moreover, the formation of $f_V$ commutes with 
finite direct sums, since so does $\bfF_T$. So $f_V$ yields
an isomorphism of functors 
$\bfF_T \cong \bfF(D_T) \otimes_{D_T} -$.
\end{proof}

\begin{remark}\label{rem:rmc}
Instead of assumption (e), we may make the stronger and much simpler
assumption that $\tcX = \cX$. (This holds for the isogeny category
of vector extensions of abelian varieties, as we will see in 
\S \ref{subsubsec:veav}). Then we obtain as in the proof of Theorem
\ref{thm:main} that \emph{$\tcF = \cF$ is hereditary, has enough 
projectives, and is equivalent to $R$-$\mod$}. 
\end{remark}

Next, we obtain a separation property of the above ring $R$,
and we describe its center $Z(R)$ as well as the center of the
abelian category $R$-mod. We denote by $Z_S$ (resp.~$Z_T$) the
center of the division ring $D_S$ (resp.~$D_T$) for any 
$S \in I$, $T \in J$.

\begin{proposition}\label{prop:center}

\begin{enumerate}

\item[{\rm (i)}] The intersection of all the left ideals of finite colength 
in $R$ is zero.

\item[{\rm (ii)}] The center $Z(R)$ consists of the triples
$(x,y,0)$, where $x = \sum_S x_S \in \bigoplus_{S \in I} Z_S$,
$y = \sum_T y_T \in \bigoplus_{T \in J} Z_T$ and $x_S m = m y_T$
for all $m \in \bfF(D_T)$.

\item[{\rm (iii)}] The center of $R$-mod is the completion of
$Z(R)$, consisting of the pairs
$(x,y)$, where $x = (x_S) \in \prod_{S \in I} Z_S$,
$y = (y_T) \in \prod_{T \in J} Z_T$ and $x_S m = m y_T$
for all $m \in \bfF(D_T)$.

\end{enumerate}

\end{proposition}

\begin{proof}
(i) Given $S \in I$ and $T \in J$, we may form the triangular matrix
ring
\[ R_{S,T} :=  \begin{pmatrix} 
D_T & 0 \\ \bfF(D_T)_S & D_S \end{pmatrix}, \]
where $\bfF(D_T)_S$ denotes the isotypical component of type $S$ 
of the $R_{\cX}$-module $\bfF(D_T)$. Clearly, $R_{S,T}$ is the
quotient of $R$ by a two-sided ideal $I_{S,T}$, and 
$\bigcap_{S,T} I_{S,T} = 0$. Thus, it suffices to show the assertion
for $R$ replaced with $R_{S,T}$. 

The left $D_S$-vector space $\bfF(D_T)_S$ contains a family of 
subspaces $(M_a)_{a \in A}$ such that the dimension of each quotient
$\bfF(D_T)/M_a$ is finite, and $\bigcap_{a \in A} M_a = 0$. Then 
$\left( \begin{smallmatrix} 0 & 0 \\ M_a & 0 \end{smallmatrix} \right)$
is a left ideal of $R_{S,T}$, as well as 
$\left( \begin{smallmatrix} D_T & 0 \\ 0 & 0 \end{smallmatrix} \right)$
and 
$\left( \begin{smallmatrix} 0 & 0 \\ 0 & D_S \end{smallmatrix} \right)$.
Moreover, all these left ideals have finite colength, and their
intersection is zero.

(ii) This is a direct verification.

(iii) Recall that the center of $R$-mod consists of the families
$z = (z_N)_{N \in R\-\mod}$ such that $z_N \in \End_R(N)$ and
$f \circ z_N = z_{N'} \circ f$ for any $f \in \Hom_R(N,N')$;
in particular, $z_N$ is central in $\End_R(N)$. Thus, $z_S \in Z_S$
and $z_T \in Z_T$ for all $S,T$. Since $N = X \oplus Y$ as a 
$R_{\cX} \oplus R_{\cY}$-module, we see that $z$ is uniquely determined 
by the families $(z_S)_{S \in I}$, $(z_T)_{T \in J}$. Moreover, we have
$z_S m = m z_T$ for all $m \in \bfF(D_T)_S$, as follows e.g. from
Lemma \ref{lem:tilde}. Thus, the center of $R$-mod is contained
in the completion of $Z(R)$. The opposite inclusion follows from (i).
\end{proof}

\begin{remark}\label{rem:center}
Assume that each $D_S$-$D_T$-bimodule $\bfF(D_T)_S$ contains a family
of sub-bimodules $(N_a)_{a \in A}$ such that each quotient
$\bfF(D_T)_S$ has finite length as a $D_S$-module, and 
$\bigcap_{a \in A} N_a = 0$. Then $R$ satisfies a stronger 
separation property, namely, its two-sided ideals of finite colength
(as left modules) have zero intersection.

The above assumption obviously holds if $\bfF(D_T)_S$ has finite
length as a $D_S$-module.It also holds if both $D_S$ and $D_T$ are 
(say) of characteristic $0$ and finite-dimensional over $\bQ$;
indeed, $\bfF(D_T)_S$ is a module over $D_S \otimes_{\bQ} D_T^{\op}$,
and the latter is a finite-dimensional semi-simple $\bQ$-algebra.
\end{remark}

\section{Applications to commutative algebraic groups}
\label{sec:acag}

\subsection{Some isogeny categories}
\label{subsec:sic}

\subsubsection{Algebraic groups \cite{DG, SGA3, Brion-II}}
\label{subsubsec:ag}

Throughout this section, we fix a ground field $k$, with
algebraic closure $\bar{k}$ and characteristic $\char(k)$. 
An \emph{algebraic group} $G$ is a group scheme of finite type 
over $k$. A \emph{subgroup} $H \subset G$ is a $k$-subgroup 
scheme; then $H$ is a closed subscheme of $G$. When 
$\char(k) = 0$, every algebraic group is smooth.

Unless otherwise mentioned, all algebraic groups will be assumed
\emph{commutative}. They form the objects 
of an abelian category $\cC$, with morphisms
the homomorphisms of $k$-group schemes (see 
\cite[VIA, Thm.~5.4.2]{SGA3}). Every object in $\cC$ is artinian
(since every decreasing sequence of closed subschemes of a scheme
of finite type eventually terminates), but generally not noetherian: 
in the multiplicative group $\bG_m$, the subgroups of roots of unity
of order $\ell^n$, where $\ell$ is a fixed prime and $n$ a non-negative
integer, form an infinite ascending chain.

The finite group schemes form a Serre subcategory $\cF$ of $\cC$. 
The quotient category $\cC/\cF$ is equivalent to the localization of 
$\cC$ with respect to the multiplicative system of \emph{isogenies}, 
i.e., of morphisms with finite kernel and cokernel. Also, $\cC/\cF$ is 
equivalent to its full subcategory $\ucC$ with objects the smooth 
connected algebraic groups (see \cite[Lem.~3.1]{Brion-II}). 
We say that $\ucC$ is the 
\emph{isogeny category of algebraic groups}. Every object of 
$\ucC$ is artinian and noetherian, i.e., $\ucC$ 
\emph{is a finite length category} (see \cite[Prop.~3.2]{Brion-II}). 

Let $G$ be an algebraic group, with group law denoted additively.
For any integer $n$, we have the multiplication map
\[n_G : G \longrightarrow G, \quad x \longmapsto n x. \]
We say that $G$ is \emph{divisible}, if $n_G$ is an epimorphism
for any $n \neq 0$. When $\char(k) = 0$, this is equivalent to 
$G$ being connected; when $\char(k) > 0$, the divisible algebraic 
groups are the semi-abelian varieties (these will be discussed 
in detail in \S \ref{subsubsec:sav}). 

If $G$ is divisible, then the natural map 
\[ \bZ \longrightarrow \End_{\cC}(G), \quad n \longmapsto n_G \]
extends to a homomorphism $\bQ \to \End_{\ucC}(G)$; in other terms,
$\End_{\ucC}(G)$ is a $\bQ$-algebra. As a consequence, 
$\Ext^n_{\ucC}(G,G')$ is a $\bQ$-vector space for any divisible 
algebraic groups $G$, $G'$ and any integer $n \geq 0$. By  
\cite[Prop.~3.6]{Brion-II}), the induced maps
\begin{equation}\label{eqn:divisible}
\Hom_{\cC}(G,G')_{\bQ} \longrightarrow
\Hom_{\ucC}(G,G'), \quad
\Ext^1_{\cC}(G,G')_{\bQ} \longrightarrow
\Ext^1_{\ucC}(G,G'),
\end{equation}
are isomorphisms, where we set $M_{\bQ} := M \otimes_{\bZ} \bQ$
for any abelian group $M$.

In particular, \emph{the isogeny category $\ucC$ is $\bQ$-linear 
when $\char(k) = 0$}; then its objects are just the connected 
algebraic groups, and its morphisms are the rational multiples of 
morphisms in $\cC$.

Given an extension of fields $K/k$ and an algebraic group $G$ 
over $k$, we obtain an algebraic group over $K$,  
\[ G_K := G \otimes_k K = G \times_{\Spec(k)} \Spec(K), \]
by extension of scalars. The assignment $G \mapsto G_K$ extends 
to the \emph{base change functor} 
\[ \otimes_k \, K: \cC = \cC_k \longrightarrow \cC_K, \]
which is exact and faithful. Also, $G$ is finite if and only if 
$G_K$ is finite. 
As a consequence, we obtain a base change functor, still denoted by 
\[ \otimes_k \, K : \ucC_k \longrightarrow \ucC_K, \] 
and which is still exact and faithful. When $K/k$ is purely inseparable, 
the above functor is an equivalence of categories
(see \cite[Thm.~3.11]{Brion-II}). We say that $\ucC_k$ is 
\emph{invariant under purely inseparable field extensions}.

Thus, to study $\ucC_k$ when $\char(k) = p > 0$, we may replace $k$ 
with its \emph{perfect closure}, 
$k_i := \bigcup_{n \geq 0} k^{1/p^n} \subset \bar{k}$, 
and hence assume that $k$ is perfect. This will be very useful,
since the structure of algebraic groups is much better understood
over a perfect ground field (see e.g. \S \ref{subsubsec:ugpc}).

\subsubsection{Linear algebraic groups, affine group schemes
\cite[III.3]{DG}}
\label{subsubsec:lag}

A (possibly non-commutative) algebraic group $G$ is called 
\emph{linear} if $G$ is isomorphic to a subgroup scheme of 
the general linear group $\GL_n$ for some integer $n > 0$; 
this is equivalent to $G$ being affine (see e.g.~\cite[Prop.~3.1.1]{Brion}). 
The smooth linear algebraic groups are the ``linear algebraic groups 
defined over $k$'' in the sense of \cite{Borel}.

For any exact sequence 
\[ 0 \longrightarrow G_1 \longrightarrow G \longrightarrow G_2 
\longrightarrow 0 \] 
in $\cC$, the group $G$ is affine if and only if $G_1$ and $G_2$ 
are affine (see e.g.~\cite[Prop.~3.1.2]{Brion}). Thus, the 
(commutative) linear algebraic groups form a Serre subcategory 
$\cL$ of $\cC$, which contains $\cF$.

The property of being affine is also invariant under field extensions 
and isogenies, in the following sense: an algebraic group $G$ is affine 
if and only if $G_K$ is affine for some field extension $K$ of $k$, 
if and only if $H$ is affine for some isogeny $f : G \to H$. 
It follows that the quotient category $\cL/\cF$ is equivalent 
to its full subcategory $\ucL$ with objects the smooth 
connected linear algebraic groups. Moreover, $\ucL$ is invariant 
under purely inseparable field extensions.

The affine $k$-group schemes (not necessarily of finite type)
form an abelian category $\tcL$, containing $\cL$ as a Serre 
subcategory. Moreover, every affine group scheme $G$ is 
the filtered inverse limit of linear algebraic groups,
quotients of $G$ (see \cite[III.3.7.4, III.3.7.5]{DG}). In fact,
$\tcL$ is the pro-completion of the abelian category $\cL$, 
in the sense of \cite[V.2.3.1]{DG}.

We say that a group scheme $G$ is \emph{pro-finite}, if
$G$ is an inverse limit of finite group schemes; equivalently,
$G$ is affine and every algebraic quotient group of $G$ is finite. 
The pro-finite group schemes form a Serre subcategory 
$\widetilde{\cF \cL}$ of $\tcL$. The quotient category 
$\tcL/\widetilde{\cF \cL}$ is the \emph{isogeny category 
of affine group schemes}; it contains $\cL/\cF$ as 
a Serre subcategory.

\subsubsection{Groups of multiplicative type 
\cite[IV.1]{DG}}
\label{subsubsec:gmt}

The invertible diagonal matrices form a subgroup scheme
$\D_n \subset \GL_n$, which is commutative, smooth 
and connected; moreover, $\D_1= \GL_1$ is isomorphic to 
the multiplicative group $\bG_m$, and $\D_n \cong \bG_m^n$ 
(the product of $n$ copies of $\bG_m$).
An algebraic group $G$ is said to be \emph{diagonalizable},
if $G$ is isomorphic to a subgroup of $\D_n$ for some $n$.
Also, $G$ is called \emph{of multiplicative type} (resp.~a 
\emph{torus}), if the base change $G_{\bar{k}}$ is diagonalizable 
(resp.~isomorphic to some $\D_{n,\bar{k}}$). Both properties are invariant
under field extensions, but not under isogenies. Also, the tori are 
the smooth connected algebraic groups of multiplicative type.
The diagonalizable algebraic groups (resp.~the algebraic groups 
of multiplicative type) form a Serre subcategory $\cD$ (resp.~$\cM$) 
of $\cC$. 

For any diagonalizable algebraic group $G$, the \emph{character group}
\[ \bfX(G) := \Hom_{\cC}(G,\bG_m) \]
is a finitely generated abelian group. Moreover, the assignment
$G \mapsto \bfX(G)$ extends to an anti-equivalence of categories
\[ \bfX : \cD \longrightarrow \bZ\-\Modfg, \]
where the right-hand side denotes the category of finitely generated 
abelian groups (see \cite[IV.1.1]{DG} for these results).

Given an algebraic group of multiplicative type $G$, there exists 
a finite Galois extension of fields $K/k$ such that $G_K$ 
is diagonalizable. Thus, $G_{k_s}$ is diagonalizable, where $k_s$ 
denotes the \emph{separable closure} of $k$ in $\bar{k}$. Let 
\[ \Gamma := \Gal(k_s/k) = \Aut(\bar{k}/k) \] 
denote the \emph{absolute Galois group} of $k$. 
Then $\Gamma$ is a pro-finite topological group, the inverse limit of 
its finite quotients $\Gal(K/k)$, where $K$ runs over the finite
Galois field extensions of $k$. Also, $\Gamma$ acts on 
the character group,
\[ \bfX(G) := \Hom_{\cC_{k_s}}(G_{k_s},\bG_{m,k_s}), \]
and the stabilizer of any character is an open subgroup. Thus,
$\bfX(G)$ is a \emph{discrete Galois module} in the sense of 
\cite[\S 2.1]{Serre-GC}. Moreover, $G$ is diagonalizable if and
only if $\Gamma$ fixes $\bfX(G)$ pointwise; then the base change map
\[ \Hom_{\cC_k}(G,\bG_{m,k}) \longrightarrow 
\Hom_{\cC_{k_s}}(G_{k_s},\bG_{m,k_s}) \]
is an isomorphism, i.e., the two notions of character groups 
are compatible. Furthermore, $G$ is a torus (resp.~finite) 
if and only if the abelian group $\bfX(G)$ is free (resp.~finite); also, 
note that the tori are the divisible algebraic groups of multiplicative 
type.

The above assignment $G \mapsto \bfX(G)$ yields an anti-equivalence
of categories
\begin{equation}\label{eqn:cartier} 
\bfX: \cM \longrightarrow \bZ \Gamma\-\Modfg 
\end{equation}
(\emph{Cartier duality}), where the right-hand side denotes 
the category of discrete $\Gamma$-modules which are finitely
generated as abelian groups. Moreover, the abelian group
$\Hom_{\cC}(T,T')$ is free of finite rank, for any tori $T$ and $T'$
(see \cite[IV.1.2, IV.1.3]{DG} for these results).

Consider the full subcategory $\cF \cM$ of $\cM$ with objects 
the finite group schemes of multiplicative type. Then $\cF \cM$ 
is a Serre subcategory of $\cM$, anti-equivalent via $\bfX$ to 
the category of finite discrete $\Gamma$-modules. Moreover, 
the quotient category $\cM/\cF \cM$ is equivalent to its full subcategory 
$\ucT$ with objects the tori, and we have an anti-equivalence 
of categories
\begin{equation}\label{eqn:qcartier} 
\bfX_\bQ : \ucT \longrightarrow \bQ \Gamma\-\mod, \quad
T \longmapsto \bfX(T)_{\bQ}.
\end{equation}
Here $\bQ \Gamma$-$\mod$ denotes the category of finite-dimensional
$\bQ$-vector spaces equipped with a discrete linear action of $\Gamma$;
note that $\bQ \Gamma$-$\mod$ is semi-simple, $\bQ$-linear and 
invariant under purely inseparable field extensions.
In view of (\ref{eqn:divisible}), this yields natural isomorphisms
\[  \Hom_{\ucT}(T,T') \cong \Hom_{\cC}(T,T')_{\bQ}
\cong \Hom^{\Gamma}(\bfX(T')_{\bQ}, \bfX(T)_{\bQ}) \]
for any tori $T$, $T'$.  As a consequence, 
\emph{the isogeny category $\ucT$ is semi-simple, 
$\bQ$-linear, Hom-finite, and invariant under purely inseparable 
field extensions}. 

Next, we extend the above results to affine $k$-group schemes,
not necessarily algebraic, by using again results of 
\cite[IV.1.2, IV.1.3]{DG}. We say that an affine group scheme 
$G$  is of multiplicative type, if so are all its
algebraic quotient groups. Denote by $\tcM$ the full subcategory
of $\tcL$ with objects the group schemes of multiplicative
type; then $\tcM$ is a Serre subcategory of $\tcL$. Moreover, 
the Cartier duality (\ref{eqn:cartier}) extends to an anti-equivalence 
of categories
\begin{equation}\label{eqn:tmult} 
\bfX : \tcM \longrightarrow \bZ \Gamma\-\Mod, 
\end{equation}
where $\bZ \Gamma$-$\Mod$ stands for the category of all discrete 
$\Gamma$-modules. Note that $\bZ \Gamma$-$\Mod$ is an abelian 
category, containing $\bZ \Gamma$-$\Modfg$ as a Serre subcategory.

Consider the full subcategory 
$\bZ \Gamma$-$\Modtors \subset \bZ \Gamma$-$\Mod$
with objects the discrete $\Gamma$-modules which are
torsion as abelian groups. Then $\bZ \Gamma$-$\Modtors$ is a Serre
subcategory of $\bZ \Gamma$-$\Mod$, anti-equivalent via $\bfX$ 
to the full subcategory $\widetilde{\cF \cM} \subset \tcM$ with objects
the pro-finite group schemes of multiplicative type.

For any $M \in \bZ \Gamma$-$\Mod$, the kernel and cokernel of
the natural map $M \to M_{\bQ}$ are torsion. It follows
readily that the induced covariant functor 
$\bZ \Gamma\-\Mod \to \bQ \Gamma\-\Mod$ yields an equivalence of
categories
\[ (\bZ \Gamma\-\Mod)/ (\bZ\Gamma\-\Modtors) 
\stackrel{\cong}{\longrightarrow} \bQ \Gamma \-\Mod, \]
where $\bQ \Gamma$-$\Mod$ denotes the category of all
$\bQ$-vector spaces equipped with a discrete linear action of
$\Gamma$. Thus, the category $\bQ \Gamma$-$\Mod$ is anti-equivalent 
to $\tcM/ \cF \tcM$. In turn, the latter category is equivalent
to its full subcategory with objects the inverse limits
of tori: the \emph{isogeny category of pro-tori}, that we denote
by $\tucT$. Clearly, the category $\bQ \Gamma$-$\Mod$ is 
semi-simple. Thus, $\tucT$ is semi-simple as well; its simple
objects are the simple tori, i.e., the tori $G$ such that every subgroup
$H \subsetneq G$ is finite.

As in Subsection \ref{subsec:rmc}, choose representatives $T$ of the
set $I$ of isomorphism classes of simple tori, and let 
$D_T := \End_{\ucT}(T)^{\op}$; then each $D_T$ is a division ring
of finite dimension over $\bQ$. Let 
\[ R_{\ucT} := \bigoplus_{T \in I} D_T, \] 
then we have an equivalence of categories 
\[ \ucT \stackrel{\cong}{\longrightarrow} R_{\ucT}\-\mod \]
which extends to an equivalence of categories
\[ \tucT \stackrel{\cong}{\longrightarrow} R_{\ucT}\-\Modss. \]
Note finally that $\tucT$ is equivalent to a Serre subcategory 
of the isogeny category $\tucL$.

\subsubsection{Unipotent groups, structure of linear groups
\cite[IV.2, IV.3]{DG}}
\label{subsubsec:ugslg}

The upper triangular matrices with all diagonal entries equal to 
$1$ form a subgroup scheme $\U_n \subset \GL_n$, which is smooth 
and connected; moreover, $\U_1$ is isomorphic to the additive
group $\bG_a$. An algebraic group $G$ is called \emph{unipotent}
if $G$ is isomorphic to a subgroup of $\U_n$ for some $n$.
The (commutative) unipotent algebraic groups form a Serre 
subcategory $\cU$ of $\cL$. Also, the property of being unipotent 
is invariant under field extensions, in the sense of \S \ref{subsubsec:lag}.

We say that an affine group scheme $G$ is unipotent, if so are
all algebraic quotients of $G$ (this differs from the definition
given in \cite[IV.2.2.2]{DG}, but both notions are equivalent
in view of \cite[IV.2.2.3]{DG}). The unipotent group schemes
form a Serre subcategory $\tcU$ of $\tcL$, which is the pro-completion
of $\cU$ (as defined in \cite[V.2.3.1]{DG}).

By \cite[IV.2.2.4, IV.3.1.1]{DG}, every affine group scheme $G$ 
lies in a unique exact sequence
\begin{equation}\label{eqn:mu} 
0 \longrightarrow M \longrightarrow G \longrightarrow U 
\longrightarrow 0, 
\end{equation}
where $M$ is of multiplicative type and $U$ is unipotent;
moreover, $\Hom_{\tcL}(M,U) = 0 = \Hom_{\tcL}(U,M)$.
Thus, $(\tcM,\tcU)$ (resp.~$(\cM,\cU)$) is a torsion pair of 
Serre subcategories of $\tcL$ (resp.~$\cL$), as considered in 
Subsection \ref{subsec:tp}.

If the field $k$ is perfect, then the exact sequence (\ref{eqn:mu}) 
has a unique splitting (see \cite[IV.3.1.1]{DG}). 
It follows that the assignment $(M,U) \mapsto M \times U$ 
yields equivalences of categories
\[ \tcM \times \tcU \stackrel{\cong}{\longrightarrow} \tcL, 
\quad  \cM \times \cU \stackrel{\cong}{\longrightarrow} \cL. \] 
In turn, this yields equivalences of isogeny categories
\begin{equation}\label{eqn:prod} 
\tucT \times \tucU \stackrel{\cong}{\longrightarrow} \tucL,
\quad
\ucT \times \ucU \stackrel{\cong}{\longrightarrow} \ucL. 
\end{equation}
In fact, \emph{the latter equivalences hold over an arbitrary field}, 
as $\ucT$, $\ucU$ and $\ucL$ are invariant under purely inseparable
field extensions. 

We now assume that $\char(k) = 0$. Then every unipotent algebraic 
group $G$ is isomorphic to the direct sum of $n$ copies of $\bG_a$, 
where $n := \dim(G)$. In particular, $G$ is isomorphic as a scheme 
to the affine space $\bA^n$, and hence is smooth and connected. 
Moreover, every morphism of unipotent groups $f : G \to H$ is linear 
in the corresponding coordinates $x_1,\ldots,x_n$ on $G$. Thus, 
\emph{the category $\cU$ is equivalent to the category $k$-$\mod$ 
of finite-dimensional $k$-vector spaces}. This extends to 
an equivalence of $\tcU$ to the category $k$-$\Mod$ of all 
$k$-vector spaces (see \cite[IV.2.4.2]{DG}).
 
As every finite unipotent group is trivial, $\cU$ and $\tcU$ are
their own isogeny categories; they are obviously semi-simple and 
$k$-linear, and $\cU$ is Hom-finite. In view of the equivalence
(\ref{eqn:prod}), it follows that 
\emph{$\ucL$ and $\tucL$ are semi-simple}.

\subsubsection{Unipotent groups in positive characteristics
\cite[V.1, V.3]{DG}}
\label{subsubsec:ugpc} 
 
Throughout this paragraph, we assume that $\char(k) = p > 0$; then every 
unipotent algebraic group $G$ is $p$-torsion. The structure of 
these groups is much more complicated than in characteristic $0$. 
For example, the additive group $\bG_a$ admits many finite subgroups, 
e.g., the (schematic) kernel of the Frobenius endomorphism 
\[ F: \bG_a \longrightarrow \bG_a, \quad x \longmapsto x^p. \]
The ring $\End_{\cU}(\bG_a)$ is generated by $k$ (acting by scalar
multiplication) and $F$, with relations $F x - x^p F = 0$ for any
$x \in k$ (see \cite[II.3.4.4]{DG}).

Assume in addition that $k$ is perfect. Then the categories 
$\cU$ and $\tcU$ may be described in terms of modules over the 
\emph{Dieudonn\'e ring} $\bD$ (see \cite[V.1]{DG}). More specifically, 
$\bD$ is a noetherian domain, generated by the ring of Witt vectors 
$W(k)$, the Frobenius $F$ and the Verschiebung $V$; also,
$R$ is non-commutative unless $k = \bF_p$. The left ideal 
$\bD V\subset \bD$ is two-sided, and the quotient ring
$\bD/\bD V$ is isomorphic to $\End_{\cU}(\bG_a)$. More generally,
for any positive integer $n$, the left ideal $\bD V^n$ is two-sided
and $\bD/\bD V^n \cong \End_{\cU}(W_n)$, where $W_n$ denotes the 
\emph{group of Witt vectors of length} $n$; this is a smooth connected
unipotent group of dimension $n$, which lies in an exact sequence
\[ 0 \longrightarrow W_n \longrightarrow W_{n+1} \longrightarrow
\bG_a \longrightarrow 0. \]
The $\End_{\cU}(\bG_a)$-module $\Ext^1_{\cU}(\bG_a,W_n)$ is freely
generated by the class of the above extension. Moreover, the
assignment 
\[ G \longmapsto \bfM(G) := \varinjlim \Hom_{\cU}(G,W_n) \]
extends to an anti-equivalence $\bfM$ of $\tcU$ with the full subcategory 
of $\bD$-$\Mod$ with objects $V$-torsion modules.
Also, $G$ is algebraic (resp.~finite) if and only if 
$\bfM(G)$ is finitely generated (resp.~of finite length); we have
$\bfM(W_n) = \bD/\bD V^n$ for all $n$. As a consequence, $\bfM$
restricts to an anti-equivalence of $\cU$ with the full subcategory of
$\bD$-$\Mod$ with objects the finitely generated modules $M$ 
which are $V$-torsion.

This yields a description of the isogeny categories $\ucU, \tucU$ 
in terms of module categories. Let $S := \bD \setminus \bD V$;
then we may form the left ring of fractions 
\[ S^{-1} \bD =: R = R_{\ucU} \]
by \cite[V.3.6.3]{DG}. This is again a (generally non-commutative) 
noetherian domain; its left ideals are the two-sided $R V^n$ in view 
of \cite[V.3.6.11]{DG}. In particular, $R$ has a unique maximal
ideal, namely, $R V$; moreover, the quotient ring
$R/ R V$ is isomorphic to the division ring of fractions 
of $\End_{\cU}(\bG_a)$. Thus, $R$ is a  discrete valuation domain
(not necessarily commutative), as considered in \cite{KT}. 
By \cite[V.3.6.7]{DG}, a morphism of unipotent group schemes 
$f : G \to H$ is an isogeny if and only if the associated morphism 
$S^{-1} \bfM(f) : S^{-1} \bfM(H) \to S^{-1} \bfM(G)$ is an isomorphism. 
As a consequence, $S^{-1} \bfM$ yields an anti-equivalence of 
$\ucU$ (resp.~$\tucU$) with $R$-$\mod$ (resp.~$R$-$\Modtors$), 
where the latter denotes the full subcategory of $R$-$\Mod$ 
with objects the $V$-torsion modules.

We now show that the abelian category $R$-$\Modtors$ is hereditary, 
and has enough projectives and a unique indecomposable projective object.
Let $M \in R$-$\Modtors$ and choose an exact sequence in $R$-$\Mod$
\[ 0 \longrightarrow M \longrightarrow I \longrightarrow J
\longrightarrow 0, \]
where $I$ is injective in $R$-$\Mod$; equivalently, the multiplication 
by $V$ in $I$ is surjective. Thus, $J$ is injective in $R$-$\Mod$ 
as well. Let $I^{\tors} \subset I$ be the largest $V$-torsion submodule. 
Then we have an exact sequence in $R$-$\Mod$
\[ 0 \longrightarrow M \longrightarrow I^{\tors} 
\longrightarrow J_{\tors} \longrightarrow 0. \]
Moreover, $I^{\tors}$, $J^{\tors}$ are injective in $R$-$\Mod$,
and hence in $R$-$\Modtors$ as well. As the injective objects
of $R$-$\Mod$ are direct sums of copies of the division ring of
fractions $K := \Fract(R)$ and of the quotient $K/R$
(see e.g.~\cite[Thm.~6.3]{KT}), it follows that the abelian category 
$R$-$\Modtors$ is hereditary and has enough injectives; moreover, 
it has a unique indecomposable injective object, namely,
\[ K/R = \varinjlim RV^{-n}/R \cong \varinjlim R/R V^n \]
(the injective hull of the simple module). Also, note that 
$R$-$\Modtors$ is equipped with a duality (i.e., an involutive
contravariant exact endofunctor), namely, the assignment
$M \mapsto \Hom_{R}(M, K/R)$. As a consequence, we obtain
an equivalence of $\tucU$ with $R$-$\Modtors$, which restricts
to an equivalence of $\ucU$ with $R$-mod.

Thus, \emph{$\tucU$ is hereditary and has enough projectives; 
its unique indecomposable projective object is $W := \varprojlim W_n$}.
Also, by \cite[V.3.6.11]{DG} (see also \cite[Thm.~4.8]{KT}), 
every unipotent algebraic group is isogenous to 
$\bigoplus_{n \geq 1} a_n W_n$ for uniquely determined integers $a_n \geq 0$. 
In other terms, \emph{every indecomposable object of $\ucU$ 
is isomorphic to $W_n$ for a unique $n \geq 1$}. 

Note finally that the above structure results for $\ucU$ extend 
to an arbitrary field $k$ of characteristic $p$, by invariance 
under purely inseparable field extensions. More specifically, 
$\ucU$ is equivalent to $R$-mod, where $R$ denotes the ring 
constructed as above from the perfect closure $k_i$.

\subsection{More isogeny categories}
\label{subsec:mic}

\subsubsection{Abelian varieties \cite{Milne-III}}
\label{subsubsec:av}

An \emph{abelian variety} is a smooth, connected algebraic group
$A$ which is proper as a $k$-scheme. Then $A$ is a projective variety 
and a divisible commutative group scheme; its group law will be denoted 
additively. Like for tori, the abelian group $\Hom_{\cC}(A,A')$ is free 
of finite rank for any abelian varieties $A$ and $A'$. Moreover, 
we have the Poincar\'e complete reducibility theorem: for any abelian
variety $A$ and any abelian subvariety $B \subset A$, there exists
an abelian subvariety $C \subset A$ such that the map
\[ B \times C \longrightarrow A, \quad (x,y) \longmapsto x + y \]
is an isogeny.
 
We denote by $\cP$ the full subcategory of $\cC$ with objects
the proper algebraic groups; then $\cP$ is a Serre subcategory
of $\cC$, containing $\cF$ and invariant under field extensions.
Moreover, the quotient category $\cP/\cF$ is equivalent to its
full subcategory $\ucA$ with objects the abelian varieties.

By (\ref{eqn:divisible}), we have an isomorphism
\[ \Hom_{\cC}(A,A')_{\bQ} \stackrel{\cong}{\longrightarrow} 
\Hom_{\ucA}(A,A') \] for any abelian varieties $A,A'$. 
Also, the base change map
\[ \Hom_{\cC_k}(A,A') \longrightarrow \Hom_{\cC_K}(A_K,A'_K) \] 
is an isomorphism for any extension of fields $K/k$ such that 
$k$ is separably closed in $K$ (see \cite[Thm.~3.19]{Conrad-II} 
for a modern version of this classical result of Chow). 

In view of the above results, 
\emph{the abelian category $\ucA$ is semi-simple, $\bQ$-linear, 
Hom-finite, and invariant under purely inseparable field extensions.
Also, $\ucA$ is a Serre subcategory of $\ucC$.} 

Like for tori again, $\ucA$ is equivalent to the category of 
all left modules of finite length over the ring
\[ R_{\ucA} := \bigoplus_{A \in J} D_A, \] 
where $J$ denotes the set of isogeny classes of simple 
abelian varieties, and we set $D_A := \End_{\ucA}(A)^{\op}$ for chosen 
representatives $A$ of the classes in $J$. Moreover, each $D_A$
is a division algebra of finite dimension over $\bQ$. Such an 
endomorphism algebra is a classical object, considered e.g. in 
\cite[Chap.~IV]{Mumford} and \cite{Oort88} where it is denoted
by $\End^0_k(A)$. The choice of a polarization of $A$ yields 
an involutory anti-automorphism of $A$ (the Rosati involution),
and hence an isomorphism of $D_A$ with its opposite algebra.

\subsubsection{General algebraic groups, quasi-compact group 
schemes \cite{Brion-II, Perrin, Perrin-II}}
\label{subsubsec:gagqc}

The linear algebraic groups form the building blocks for all
connected algebraic groups, together with the abelian varieties.
Indeed, we have \emph{Chevalley's structure theorem}: for any 
connected algebraic group $G$, there exists an exact sequence 
\[ 0 \longrightarrow L \longrightarrow G  \longrightarrow A 
 \longrightarrow 0, \]
where $L$ is linear and $A$ is an abelian variety. Moreover, 
there is a unique smallest such subgroup $L \subset G$, and 
this group is connected. If $G$ is smooth and $k$ is perfect, 
then $L$ is smooth as well (see \cite{Conrad, Brion} for modern 
expositions of this classical result). 

Returning to an arbitrary ground field $k$, it is easy to see that
$\Hom_{\cC}(A,L) = 0$ for any abelian variety $A$ and
any linear algebraic group $L$; also, the image of any 
morphism $L \to A$ is finite (see e.g.~\cite[Prop.~2.5]{Brion-II}).

It follows that $(\ucL,\ucA)$ is a torsion pair of Serre subcategories in 
$\ucC$, and we have $\Hom_{\ucC}(A,L) = 0$ for all $A \in \ucA$, $L \in \ucL$.
Therefore, $\Ext^1_{\ucC}(L,A) = 0$ for all such $A$, $L$ by Lemma 
\ref{lem:exact}. In view of Chevalley's structure theorem and the 
vanishing of $\Ext^1_{\ucC}(A',A)$ for all $A,A' \in \ucA$, we obtain 
that  $\Ext^1_{\ucC}(G,A) = 0$ for all $G \in \ucC$ and $A \in \ucA$.
Thus, \emph{every abelian variety is injective in $\ucC$} 
(see \cite[Thm.~5.16]{Brion-II} for the determination of the injective 
objects of $\ucC$). 

By (\ref{eqn:prod}), we have $\Ext^1_{\ucC}(T,U) = 0 = \Ext^1_{\ucC}(U,T)$ 
for all $T \in \ucT$, $U \in \ucU$. Also, recall that $\Ext^1_{\ucC}(T,A) = 0$ 
for all $A \in \ucA$ and $\Ext^1_{\ucC}(T,T') = 0$ for all $T' \in \ucT$. 
By Chevalley's structure theorem again, it follows that 
$\Ext^1_{\ucC}(T,G) = 0$ for all $G \in \ucC$. Thus, 
\emph{every torus is projective in $\ucC$}. If $\char(k) = 0$, then
$\ucL$ is semi-simple, as seen in \S \ref{subsubsec:ugslg}.
In view of Corollary \ref{cor:hereditary}, it follows that every 
linear algebraic group is projective in $\ucC$ (see 
\cite[Thm.~5.14]{Brion-II} for the determination of the projective 
objects of $\ucC$ in arbitrary characteristics).

We now adapt the above results to the setting of \emph{quasi-compact}
group schemes. Recall that a scheme is quasi-compact if every open
covering admits a finite refinement. Every affine scheme is quasi-compact,
as well as every scheme of finite type (in particular, every
algebraic group). Also, every connected group scheme is quasi-compact
(see \cite[II.2.4, II.2.5]{Perrin} or \cite[VIA, Thm.~2.6.5]{SGA3}). 
The quasi-compact (commutative) group schemes form an abelian category 
$\tcC$, containing $\cC$ as a Serre subcategory. Moreover, every 
$G \in \tcC$ is the limit of a filtered inverse system 
$((G_i)_{i \in I}, (u_{ij} : G_j \to G_i)_{i \leq j})$
such that the $G_i$ are algebraic groups and the $u_{ij}$ are affine
morphisms (see \cite[V.3.1, V.3.6]{Perrin}). Also, there is a unique
exact sequence in $\tcC$
\[ 0 \longrightarrow G^0 \longrightarrow G \longrightarrow F
\longrightarrow 0, \]
where $G^0$ is connected and $F$ is \emph{pro-\'etale} 
(i.e., a filtered inverse limit of finite \'etale group schemes); see 
\cite[II.2.4, V.4.1]{Perrin}. Finally, there is an exact sequence
as in Chevalley's structure theorem
\[ 0 \longrightarrow H \longrightarrow G^0 \longrightarrow A
\longrightarrow 0, \]
where $H$ is an affine group scheme, and $A$ an abelian variety
(see \cite[V.4.3.1]{Perrin}).

Note that $\tcC$ is not the pro-completion of the abelian category
$\cC$, as infinite products do not necessarily exist in $\tcC$.
For example, the product of infinitely many copies of a non-zero 
abelian variety $A$ is not represented by a scheme (this may be checked 
by arguing as in \cite[91.48]{SP}, with the morphism $\SL_2 \to \bP^1$
replaced by the surjective smooth affine morphism $V \to A$, 
where $V$ denotes the disjoint union of finitely many open affine 
subschemes covering $A$).  

We define the \emph{isogeny category of quasi-compact group schemes},
$\tucC$, as the quotient category of $\tcC$ by the Serre subcategory
$\widetilde{\cF \cC} = \widetilde{\cF \cL}$ of pro-finite group schemes. 
Every object of $\tucC$ is isomorphic to an extension of an abelian variety 
$A$ by an affine group scheme $H$. Moreover, $\Hom_{\tucC}(A,H) = 0$
and the image of every morphism $f : H \to A$ is finite
(indeed, $f$ factors through a closed immersion $H/\Ker(f) \to A$
by \cite[V.3.3]{Perrin}). As a consequence, $(\tucL,\ucA)$ is a
torsion pair of Serre subcategories of $\tucC$; moreover, 
$\Hom_{\tucC}(A,H) = 0$ for all $A \in \ucA$, $H \in \tucL$.

Like for the category $\ucC$, it follows that every abelian variety
is projective in $\tucC$, and every pro-torus is injective; when
$\char(k) = 0$, every affine group scheme is projective.

\subsubsection{Vector extensions of abelian varieties
\cite[5.1]{Brion-II}}
\label{subsubsec:veav}

The objects of the title are the algebraic groups $G$ obtained as 
extensions
\begin{equation}\label{eqn:vext}
0 \longrightarrow U \longrightarrow G \longrightarrow A \longrightarrow 0, 
\end{equation}
where $A$ is an abelian variety and $U$ is a \emph{vector group},
i.e., $U \cong n \bG_a$ for some $n$. As
$\Hom_{\cC}(U,A) = 0 = \Hom_{\cC}(A,U)$ 
(see \S \ref{subsubsec:lag}), the data of $G$ and of the extension
(\ref{eqn:veav}) are equivalent. Also, we have a bi-functorial isomorphism
\begin{equation}\label{eqn:rs}
\Ext^1_{\cC}(A,\bG_a) \stackrel{\cong}{\longrightarrow}
H^1(A,\cO_A),
\end{equation}
where the right-hand side is a $k$-vector space of
dimension $\dim(A)$ (see \cite[III.17]{Oort66}).

If $\char(k)= p > 0$, then $p_U = 0$ and hence the class of
the extension (\ref{eqn:vext}) is killed by $p$. Thus, this
extension splits after pull-back by the isogeny $p_A : A  \to A$.

From now on, we assume that $\char(k) = 0$; then the vector
extensions of $A$ are the extensions by unipotent groups. 
We denote by $\ucV$ the full subcategory of $\ucC$
with objects the vector extensions of abelian varieties.
By the Chevalley structure theorem (\S \ref{subsubsec:lag})
and the structure of linear algebraic groups 
(\S \ref{subsubsec:ugslg}), a connected algebraic group $G$ is 
an object of $\ucV$ if and only if $\Hom_{\ucC}(T,G) = 0$.
As the functor $\Hom_{\ucC}(T,?)$ is exact, it follows
that $\ucV$ is a Serre subcategory of $\ucC$. Moreover, 
$(\cU,\ucA)$ is a torsion pair of Serre subcategories of $\ucV$; 
they are both semi-simple in view of \S\S \ref{subsubsec:ugslg} 
and \ref{subsubsec:av}. 

By (\ref{eqn:divisible}), we have an isomorphism
\[ \Ext^1_{\cC}(A,U)_{\bQ}
\stackrel{\cong}{\longrightarrow} \Ext^1_{\ucC}(A,U). \]
In view of (\ref{eqn:rs}), this yields bi-functorial isomorphisms
\begin{equation}\label{eqn:veav} 
\Ext^1_{\ucV}(A,U) \cong 
H^1(A,\cO_A) \otimes_k U \cong
\Hom_k(H^1(A,\cO_A)^*, U), 
\end{equation}
where $H^1(A,\cO_A)^*$ denotes of course the dual $k$-vector space. 
Moreover, the assignment $A \mapsto H^1(A,\cO_A)^*$ extends to 
a covariant exact functor
\[ \bfU : \ucA \longrightarrow \cU, \]
as follows e.g.~from \cite[Cor.~5.3]{Brion-II}. 

So the triple $(\ucV,\cU,\ucA)$ satisfies the assumptions 
(a), (b), (c), (d) of Subsection \ref{subsec:ue}, with $\tcX = \cU$ 
and $\bfF = \bfU$. With the notation of \S \ref{subsubsec:av}, set 
\[ M_A := H^1(A,\cO_A)^* \]
for any $A \in J$; then $M_A$ is a $k$-$D_A$-bimodule.

\begin{proposition}\label{prop:veav}
With the above notation, there is an equivalence of categories 
\[ \bfM_{\cV} : \ucV \stackrel{\cong}{\longrightarrow}
R_{\ucV}\-\mod, \]
where $R_{\ucV}$ stands for the triangular matrix ring
\[ \begin{pmatrix}
\bigoplus_{A \in J} D_A & 0 \\
\bigoplus_{A \in J} M_A & k \\
\end{pmatrix}. \]
Moreover, the center of $\ucV$ is $\bQ$.
\end{proposition}

\begin{proof}
The first assertion follows from Theorem \ref{thm:equiv} 
and Remark \ref{rem:rmc}.

By Proposition \ref{prop:center}, the center of $\ucV$ consists
of the pairs $z = (x,(y_A)_{A \in J})$ where $x \in k$, $y_A \in Z_A$ 
(the center of $D_A$) and $x m = m y_A$ for all $m \in M_A$ and 
$A \in J$. In particular, if the simple abelian variety $A$ 
satisfies $D_A = \bQ$, then $x = y_A \in \bQ$. 
As such abelian varieties exist (e.g., elliptic curves without 
complex multiplication), it follows that $x \in \bQ$. 
Then for any $A \in J$, we obtain $y_A = x$ as $M_A \neq 0$. 
So $z \in \bQ$.
\end{proof}

Next, recall that every abelian variety $A$ has a universal 
vector extension $\bfE(A)$, by the vector group $\bfU(A)$.
In view of Proposition \ref{prop:universal}, the projective objects
of $\ucV$ are the products of unipotent groups and
universal vector extensions; moreover, every 
$G \in \ucV$ has a canonical projective resolution,
\begin{equation}\label{eqn:presv} 
0 \longrightarrow \bfU(A) \longrightarrow 
U \times \bfE(A) \longrightarrow G \longrightarrow 0, 
\end{equation}
where $A$ denotes of course the abelian variety quotient of $G$. 
In particular, the abelian category $\ucV$ is hereditary and has
enough projectives. This recovers most of the results in
\cite[Sec.~5.1]{Brion-II}.

\subsubsection{Semi-abelian varieties \cite[5.4]{Brion}}
\label{subsubsec:sav}

These are the algebraic groups $G$ obtained as extensions
\begin{equation}\label{eqn:sav}
0 \longrightarrow T \longrightarrow G \longrightarrow A \longrightarrow 0, 
\end{equation}
where $A$ is an abelian variety and $T$ is a torus.
Like for vector extensions of abelian varieties, we have 
$\Hom_{\cC}(T,A) = 0 = \Hom_{\cC}(A,T)$; thus, the data of $G$ and 
of the extension (\ref{eqn:sav}) are equivalent. 

The Weil-Barsotti formula (see e.g.~\cite[III.17, III.18]{Oort66}) 
yields a bi-functorial isomorphism
\[ \Ext^1_{\cC}(A,T) \stackrel{\cong}{\longrightarrow}
\Hom^{\Gamma}(\bfX(T), \wA(k_s)). \]
Here $\wA$ denotes the dual of $A$; this is an abelian variety
with dimension $\dim(A)$ and with Lie algebra $H^1(A,\cO_A)$.
In view of (\ref{eqn:divisible}), this yields in turn a bi-functorial 
isomorphism
\begin{equation}\label{eqn:wb}
\Ext^1_{\ucC}(A,T) \stackrel{\cong}{\longrightarrow}
\Hom^{\Gamma}(\bfX(T), \wA(k_s)_{\bQ}).
\end{equation}

We denote by $\ucS$ the full subcategory of $\ucC$ with 
objects the semi-abelian varieties. By 
\cite[Lem.~5.4.3, Cor.~5.4.6]{Brion}, $\ucS$ is a Serre subcategory 
of $\ucC$, invariant under purely inseparable field extensions. 
Moreover, $(\ucT,\ucA)$ is a torsion pair of Serre subcategories of 
$\ucS$. The assignment $A \mapsto \wA(k_s)_{\bQ}$ extends
to a contravariant exact functor $\ucA \to \bQ \Gamma$-$\Mod$
in view of \cite[Rem.~4.8]{Brion-II} (see also Lemma 
\ref{lem:points} below). Since $\bQ \Gamma$-$\Mod$
is anti-equivalent to $\tucT$ (\S \ref{subsubsec:gmt}), this yields 
a covariant exact functor
\[ \tbfT : \ucA \longrightarrow \tucT \]
together with a bi-functorial isomorphism
\[ \Ext^1_{\ucC}(A,T) \stackrel{\cong}{\longrightarrow}
\Hom_{\tucT}(\tbfT(A),T). \]
Thus, the triple $(\ucS,\ucT,\ucA)$ satisfies the assumptions
(a), (b), (c), (d) of Subsection \ref{subsec:ue} with $\tcX = \tucT$
and $\bfF = \tbfT$; moreover, the assumption (e) of Subsection 
\ref{subsec:rmc} holds by construction. With the notation of 
\S\S \ref{subsubsec:gmt} and \ref{subsubsec:av}, let 
\[ M_{T,A} := \Hom^{\Gamma}(\bfX(T), \wA(k_s)_{\bQ}) \] 
for all $T \in I$, $A \in J$; then $M_{T,A}$ is a $D_T$-$D_A$-bimodule.
In view of Theorem \ref{thm:equiv}, we obtain:

\begin{proposition}\label{prop:sav}
There is an equivalence of categories
\[\bfM_{\cS} : \ucS \stackrel{\cong}{\longrightarrow}
R_{\ucS}\-\mod, \]
where $R_{\ucS}$ stands for the triangular matrix ring
\[ \begin{pmatrix}
\bigoplus_{A \in J} D_A & 0 \\
\bigoplus_{T \in I, A \in J} M_{T,A} & \bigoplus_{T \in I} D_T. \\
\end{pmatrix}. \]

\end{proposition}

\begin{remark}\label{rem:sav}
If $k$ is locally finite (i.e., the union of its finite subfields), 
then the abelian group $A(k)$ is torsion for any abelian variety $A$. 
It follows that $\ucS \cong \ucT \times \ucA$ and 
$R_{\ucS} = R_{\ucT} \times R_{\ucA}$. In particular, the center of 
$R_{\ucS}$ is an infinite direct sum of fields.
 
On the other hand, if $k$ is not locally finite, then the abelian 
group $A(k_s)$ has infinite rank for any non-zero abelian variety $A$
(see  \cite[Thm.~9.1]{FJ}). As a consequence, $A$ admits no universal 
extension in $\ucS$. If in addition $k$ is separably closed, then
the center of $R_{\ucS}$ is $\bQ$, as follows by arguing as in the
proof of Proposition \ref{prop:veav} with $\bG_a$ replaced by $\bG_m$.
We do not know how to determine the center of $R_{\ucS}$ for an
arbitrary (not locally finite) field $k$.
\end{remark}

Next, consider the isogeny category $\tucC$ of quasi-compact
group schemes, and denote by $\tucS \subset \tucC$ the full 
subcategory with objects the group schemes obtained 
as extensions of abelian varieties by pro-tori. 
Then $\ucS$ is a Serre subcategory of $\tucS$;
moreover, $(\tucT,\ucA)$ is a torsion pair of Serre subcategories
of $\tucS$, and $\Hom_{\tucS}(A,T) = 0$ for any abelian variety $A$
and any pro-torus $T$. Thus, $\tucS$ is equivalent to the category of 
extensions of abelian varieties by pro-tori (as defined in Subsection 
\ref{subsec:ce}), and in turn to the category $R_{\ucS}$-$\tmod$ by 
Theorem \ref{thm:equiv}. 

In view of the results of Subsection \ref{subsec:ue}, 
every abelian variety $A$ has a universal extension in $\tucS$, 
by the pro-torus with Cartier dual $\wA(k_s)_{\bQ}$. 
Moreover, the projective objects of $\tucS$ are the products of
pro-tori and universal extensions; also, every 
$G \in \tucS$ has a canonical projective resolution, similar to
(\ref{eqn:presv}). In particular, the abelian category $\tucS$ is
hereditary and has enough projectives.

\subsubsection{General algebraic groups (continued)}
\label{subsubsec:gagc}

We return to the setting of \S \ref{subsubsec:gagqc}, and consider
the isogeny category of algebraic groups, $\ucC$, as a Serre category 
of the isogeny category of quasi-compact group schemes, $\tucC$.
Also, recall from \S\S \ref{subsubsec:ugslg}, \ref{subsubsec:ugpc} 
the isogeny category of unipotent algebraic groups, $\ucU$,
a Serre subcategory of that of unipotent group schemes, $\tucU$.
Likewise, we have the isogeny category of semi-abelian varieties,
$\ucS$, a Serre subcategory of the isogeny category of extensions 
of abelian varieties by pro-tori, $\tucS$. These are the ingredients 
of a structure result for $\ucC$, $\tucC$ in positive characteristics:

\begin{proposition}\label{prop:charp}
If $\char(k) = p > 0$, then the assignment $(S,U) \mapsto S \times U$
extends to equivalences of categories
\[ \ucS \times \ucU \stackrel{\cong}{\longrightarrow} \ucC, \quad
 \tucS \times \tucU \stackrel{\cong}{\longrightarrow} \tucC. \]
\end{proposition}

\begin{proof}
The first equivalence is obtained in \cite[Prop.~5.10]{Brion-II}.
We provide an alternative proof: by Chevalley's structure theorem,
every $G \in \ucC$ lies in an exact sequence in $\ucC$
\[ 0 \longrightarrow U \longrightarrow G \longrightarrow S
\longrightarrow 0, \]
where $U \in \ucU$ and $S \in \ucS$. Moreover, we have
$\Ext^n_{\ucC}(U,S) = 0 = \Ext^n_{\ucC}(S,U)$ for all $n \geq 0$,
since the multiplication map $p_S$ is an isomorphism in $\ucC$,
while $p_U^n = 0$ for $n \gg 0$. In particular, the above exact
sequence has a unique splitting, which is functorial in $U$, $S$.

The second equivalence follows from the first one, as every
quasi-compact group scheme is the inverse limit of its algebraic
quotient groups.
\end{proof}
 
Next, we obtain the main result of this paper; for this, we gather 
some notation. Define a ring $R = R_{\ucC}$ by 
\[ R = 
\begin{cases}
R_{\ucS} \times R_{\ucU}, & \text{if } \char(k) > 0 \\
R_{\ucS} \times_{R_{\ucA}} R_{\ucV}, & \text{if } \char(k) = 0.
\end{cases} \]
More specifically, 
\[ R = \begin{pmatrix}
\bigoplus_{A \in J} D_A & 0 \\
\bigoplus_{T \in I, A \in J} M_{T,A} & \bigoplus_{T \in I} D_T \\
\end{pmatrix} 
\times (\bD \setminus \bD V)^{-1} \bD \]
if $\char(k) > 0$, where $I$ (resp.~$J$) denotes the set 
of isogeny classes of simple tori (resp.~of simple abelian 
varieties), $T \in I$ (resp.~$A \in J$) denote representatives 
of their classes, $D_T := \End_{\cT}(T)_{\bQ}^{\op}$,  
$D_A := \End_{\cA}(A)_{\bQ}^{\op}$, and 
$M_{T,A} := \Hom^{\Gamma}(\bfX(T), \wA(k_s)_{\bQ})$. Moreover,
$\bD$ denotes the Dieudonn\'e ring over the perfect
closure of $k$, and $V \in \bD$ the Verschiebung as in 
\S \ref{subsubsec:ugpc}. 

If $\char(k) = 0$, then
\[ R = \begin{pmatrix}
\bigoplus_{A \in J} D_A & 0 \\
(\bigoplus_{T \in I, A \in J} M_{T,A}) \oplus (\bigoplus_{A \in J} M_A)
& (\bigoplus_{T \in I} D_T) \oplus k \\
\end{pmatrix}, \]
where $M_A := H^1(A,\cO_A)^*$. We may now state:

\begin{theorem}\label{thm:main}
With the above notation, the abelian categories $\ucC$ and $\tucC$ 
are hereditary, and $\tucC$ has enough projectives. Moreover,
there are compatible equivalences of categories
\[ \ucC \stackrel{\cong}{\longrightarrow} R\-\mod, \quad 
\tucC \stackrel{\cong}{\longrightarrow} R\-\tmod. \]

\end{theorem}

\begin{proof}
When $\char(k) > 0$, this follows by combining Proposition
\ref{prop:charp} with the structure results for $\ucU$, $\tucU$
recalled in \S \ref{subsubsec:ugpc}, and those for $\ucS$,
$\tucS$ obtained in Proposition \ref{prop:sav}. When 
$\char(k) = 0$, recall from \S \ref{subsubsec:gagqc} that 
$(\ucL,\ucA)$ (resp.~$(\tucL,\ucA)$) is a torsion pair of Serre
subcategories of $\ucC$ (resp.~$\tucC$); moreover,
$\ucL \cong \ucT \times \cU$ and $\tucL \cong \tucT \times \tcU$
by \S \ref{subsubsec:ugslg}. In view of the bi-functorial isomorphisms
(\ref{eqn:veav}) and (\ref{eqn:wb}), the assertions follow from 
Theorem \ref{thm:equiv} like in the cases of vector extensions of
abelian varieties (Proposition \ref{prop:veav}) and of semi-abelian 
varieties (Proposition \ref{prop:sav}).
\end{proof}

Also, recall from Proposition \ref{prop:center} that the intersection
of the left ideals of finite colength in $R$ is zero. One may check
(by using Remark \ref{rem:center}) that this also holds for the 
two-sided ideals of finite colength as left modules. 

We will show in Subsection \ref{subsec:fcheg} that the center of 
$\ucC$ equals $\bQ$ if $\char(k) = 0$, and contains $\bQ \times \bZ_p$
if $\char(k) = p > 0$.

\subsection{Functors of points}
\label{subsec:fp}

Let $G$ be an algebraic group. Then the group of $\bar{k}$-points,
$G(\bar{k})$, is equipped with an action of the absolute Galois
group $\Gamma$. We also have the subgroup of $k_s$-points, 
$G(k_s)$, which is stable under $\Gamma$. Since every 
$x \in G(k_s)$ lies in $G(K)$ for some finite Galois extension of fields
$K/k$, we see that the stabilizer of $x$ in $\Gamma$
is open, i.e., $G(k_s)$ is a discrete $\Gamma$-module. Likewise,
the $\Gamma$-module $G(\bar{k})$ is discrete as well.

Clearly, the assignment $G \mapsto G(\bar{k})$ extends to a
covariant exact functor
\[ (\bar{k}) : \cC \longrightarrow \bZ \Gamma\-\Mod, \quad
\cF \longrightarrow \bZ \Gamma \- \Modtors. \]
The assignment $G \mapsto G(k_s)$ also extends to a covariant functor 
\[ (k_s) : \cC \longrightarrow \bZ \Gamma\-\Mod, \quad
\cF \longrightarrow \bZ \Gamma\-\Modtors \]
which is additive and left exact. But the functor $(k_s)$ is not exact
when $k$ is an imperfect field, as seen from the exact sequence
\[ 0 \longrightarrow \alpha_p \longrightarrow \bG_a
\stackrel{F}{\longrightarrow} \bG_a \longrightarrow 0, \]
where $p := \char(k)$ and $F$ denotes the Frobenius endomorphism, 
$x \mapsto x^p$. Yet the functors $(\bar{k})$ and $(k_s)$ are 
closely related:

\begin{lemma}\label{lem:exponent}

\begin{enumerate}

\item[{\rm (i)}]
Let $K/k$ be an extension of fields of characteristic $p > 0$
and assume that $K^{p^n}\subset k$. Then $p^n x \in G(k)$ 
for any $x \in G(K)$.

\item[{\rm (ii)}] For any $x \in G(\bar{k})$, there exists
$n = n(x) \geq 0$ such that $p^n x \in G(k_s)$.

\end{enumerate}

\end{lemma}

\begin{proof}
(i) By assumption, we have $x \in G(k^{1/p^n})$. Consider the $n$th 
relative Frobenius morphism,
\[ F^n_{G/k} : G \longrightarrow G^{(p^n)} \] 
(see e.g.~\cite[A.3]{CGP}). Then $F^n_{G/k}(x) \in G^{(p^n)}(k)$:
indeed, this holds with $G$ replaced with any scheme of finite type,
since this holds for the affine space $\bA^n$ and the formation
of the relative Frobenius morphism commutes with immersions.
We also have the $n$th Verschiebung, 
\[ V^n_{G/k} : G^{(p^n)} \longrightarrow G, \]
which satisfies $V^n_{G/K} \circ F^n_{G/k} = p^n_G$ (see 
\cite[VIIA.4.3]{SGA3}). It follows that
$p^n x = V^n_{G/k}(F^n_{G/k}(x))$ is in $G(k)$.

(ii) Just apply (i) to $k_s$ instead of $k$, and use the fact that
$\bar{k} = \bigcup_{n \geq 0} k_s^{1/p^n}$.
\end{proof}

As a direct consequence, we obtain:

\begin{lemma}\label{lem:points}

\begin{enumerate}

\item[{\rm (i)}] The natural map $G(k_s)_{\bQ} \to G(\bar{k})_{\bQ}$
is an isomorphism for any algebraic group $G$.

\item[{\rm (ii)}] The covariant exact functor
\[ (\bar{k})_{\bQ}: \cC \longrightarrow \bQ \Gamma\-\Mod, \quad 
G \longmapsto G(\bar{k})_{\bQ} := G(\bar{k})_{\bQ} \]
yields a covariant exact functor, also denoted by   
\[ (\bar{k})_{\bQ}: \ucC \longrightarrow \bQ \Gamma\-\Mod. \]

\end{enumerate}

\end{lemma}

\begin{remarks}\label{rem:points}
(i) Assume that $G$ is unipotent. If $\char(k) = p > 0$,
then $G(\bar{k})$ is $p$-torsion and hence $G(\bar{k})_{\bQ} = 0$.
On the other hand, if $\char(k) = 0$, then $G \cong n \bG_a$ 
and hence $G(\bar{k})_{\bQ} \cong n \bar{k}$ as a $\Gamma$-module.
Using the normal basis theorem, it follows that the multiplicity 
of any simple discrete $\Gamma$-module $M$ in $G(\bar{k})_{\bQ}$ 
equals $n \dim(M)$; in particular, all these multiplicities are finite.

The latter property does not extend to the case where $G$ is 
a torus. For example, the multiplicity of the trivial $\Gamma$-module
in $\bG_m(\bar{k})_{\bQ}$ is the rank of the multiplicative group
$k^*$, which is infinite when $k$ is not locally finite. Indeed, 
under that assumption, $k$ contains either $\bQ$ or the field 
of rational functions in one variable $\bF_p(t)$; so the assertion
follows from the infiniteness of prime numbers and of irreducible 
polynomials in $\bF_p[t]$.

Whenn $G$ is an abelian variety, the finiteness of multiplicities 
of the $\Gamma$-module $G(k)_{\bQ}$ will be discussed in the
next subsection.

\medskip

\noindent
(ii) For any abelian variety $A$ and any torus $T$, we have 
a bi-functorial isomorphism
\[ \Ext^1_{\ucC}(A,T) \stackrel{\cong}{\longrightarrow}
\Hom^{\Gamma}(\bfX(T), \wA(\bar{k})_{\bQ}), \]
as follows from (\ref{eqn:wb}) together with Lemma 
\ref{lem:points}. Also, recall that $\wA$ is isogenous to $A$ 
(see e.g.~\cite{Milne-III}) and hence the $\Gamma$-module 
$\wA(\bar{k})_{\bQ}$ is isomorphic (non-canonically) to 
$A(\bar{k})_{\bQ}$. 
\end{remarks}

Next, we associate an endofunctor of $\ucC$ with any discrete 
$\Gamma$-module $M$, which is a free abelian group of finite
rank. We may choose a finite Galois extension 
$K/k \subset k_s/k$ such that $\Gamma$ acts on $M$ via 
its finite quotient $\Gamma' := \Gal(K/k)$. For any 
$G \in \cC$, consider the tensor product of commutative group 
functors $G_K \otimes _{\bZ} M$. This group functor is 
represented by an algebraic group over $K$ (isomorphic to 
the product of $r$ copies of $G_K$, where $r$ denotes the rank 
of $M$ as an abelian group), equipped with an action of $\Gamma'$ 
such that the structure map $G_K \otimes_{\bZ} M \to \Spec(K)$ 
is equivariant. By Galois descent (see e.g.~\cite[Cor.~3.4]{Conrad-II}), 
the quotient
\[ G(M) := (G_K \otimes_{\bZ} M)/\Gamma' \]
is an algebraic group over $k$, equipped with a natural
$\Gamma'$-equivariant isomorphism
\[ G_K \otimes_{\bZ} M 
\stackrel{\cong}{\longrightarrow} G(M)_K. \]
The assignment $G \mapsto G(M)$ extends to a covariant
endofunctor $(M)$ of $\cC$, which is exact as 
the base change functor $\otimes_k \, K$ is faithful. 
Moreover, one may easily check that $(M)$ is independent
of the choice of $K$, and hence comes with a natural 
$\Gamma$-equivariant isomorphism
\[ G_{k_s} \otimes_{\bZ} M 
\stackrel{\cong}{\longrightarrow} G(M)_{k_s} \]
for any $G \in \cC$. In particular, we have an isomorphism
of $\Gamma$-modules 
\[ G(k_s) \otimes_{\bZ} M \stackrel{\cong}{\longrightarrow} 
G(M)(k_s). \]
This implies readily:

\begin{lemma}\label{lem:twist}
With the above notation and assumptions, the endofunctor
$(M)$ of $\cC$ yields a covariant exact endofunctor 
$(M)_{\bQ}$ of $\ucC$ that stabilizes $\ucT$, $\ucU$, $\ucA$
and $\ucS$. Moreover, $(M)_{\bQ}$ only depends on 
$M_\bQ \in \bQ\Gamma$-$\mod$, and there is a natural isomorphism 
of $\Gamma$-modules
\[ 
G(M)(k_s)_{\bQ} \cong G(k_s) \otimes_{\bZ} M_{\bQ} = 
G(k_s)_{\bQ} \otimes_{\bQ} M_{\bQ} 
\]
for any $G \in \cC$.
\end{lemma}

\begin{remarks}\label{rem:twist}
(i) If $G$ is a torus, then one readily checks that $G(M)$ is 
the torus with character group $\Hom_{\bZ}(M, \bfX(G))$. 
As a consequence, the endofunctor $(M)$ of $\ucT$ is identified 
with the tensor product by the dual module $M^*$, under the 
anti-equivalence of categories of $\ucT$ with $\bQ \Gamma$-mod.

On the other hand, the endofunctor $(M)$ of $\ucU$ is just given by
the assignment $G \mapsto r G$, where $r$ denotes the rank of the 
free abelian group $M$. Indeed, we may assume that $G$ is indecomposable
and (using the invariance of $\ucU$ under purely inseparable field
extensions) that $k$ is perfect. Then $G \cong W_n$ and hence
$G(M)_{k_s} = W_{n,k_s} \otimes_{\bZ} M \cong r W_{n,k_s}$
in $\ucU_{k_s}$. In view of the uniqueness of the decomposition
in $\ucU$ as a direct sum of groups of Witt vectors, it follows that
$G(M) \cong r W_n$ as desired.

\medskip

\noindent
(ii) The endofunctor $(M)$ can be interpreted in terms 
of Weil restriction when $M$ is a permutation $\Gamma$-module,
i.e., $M$ has a $\bZ$-basis which is stable under $\Gamma$. 
Denote by $\Delta \subset \Gamma$ the isotropy group of some basis element,
and by $K \subset k_s$ the fixed point subfield of $\Delta$. Then
$K/k$ is a finite separable field extension, and one may check 
that there is a natural isomorphism $G(M) \cong \R_{K/k}(G_K)$ 
with the notation of \cite[A.5]{CGP}.

\medskip

\noindent
(iii) The assignment $G \mapsto G(M)$ is in fact a special case of a tensor 
product construction introduced by Milne in the setting of abelian varieties
(see \cite{Milne-II}) and systematically studied by Mazur, Rubin and Silverberg 
in \cite{MRS}. More specifically, the tensor product $M \otimes_{\bZ} G$, 
defined there in terms of Galois cohomology, is isomorphic to $G(M)$ in view 
of \cite[Thm.~1.4]{MRS}. 
\end{remarks}

\subsection{Finiteness conditions for Hom and Ext groups}
\label{subsec:fcheg}

Recall from \S \S \ref{subsubsec:gmt}, \ref{subsubsec:av} 
that the abelian categories $\ucT$, $\ucA$ are $\bQ$-linear and
Hom-finite. Also, recall from \S \ref{subsubsec:ugslg} that 
$\ucU \cong \cU \cong k\-\mod$ is $k$-linear, semi-simple and
Hom-finite when $\char(k) = 0$.

\begin{proposition}\label{prop:qlin}

\begin{enumerate}

\item[{\rm (i)}] The abelian categories $\ucT$ and $\ucA$ are not 
$K$-linear for any field $K$ strictly containing $\bQ$.

\item[{\rm (ii)}] The abelian category $\ucU$ is not $K$-linear for
any field $K$, when $\char(k) = p > 0$.

\item[{\rm (iii)}] The center of $\ucC$ is $\bQ$ when $\char(k) = 0$.
\end{enumerate}

\end{proposition}

\begin{proof}
(i) The assertion clearly holds for $\ucT$, as 
$\End_{\ucT}(\bG_m) = \bQ$. For $\ucA$, we replace $\bG_m$ with
appropriate elliptic curves $E$. Given any $t \in k$,
there exists such a curve with $j$-invariant $t$ (see 
e.g.~\cite[Prop.~III.1.4]{Silverman}). If $\char(k) = 0$, then we choose
$t \in \bQ \setminus \bZ$; in particular, $t$ is not an algebraic
integer. By \cite[Thm.~C.11.2]{Silverman}, we have 
$\End_{\cA}(E) = \bZ$ and hence $\End_{\ucA}(E) = \bQ$. If 
$\char(k) = p$ and $k$ is not algebraic over $\bF_p$, then we 
may choose $t$ transcendental over $\bF_p$. By 
\cite[\S 22]{Mumford}, we then have again $\End_{\ucA}(E) = \bQ$. 
Finally, if $k$ is algebraic over $\bF_p$, then every abelian
variety $A$ is defined over a finite subfield of $k$, and hence
the associated Frobenius endomorphism lies in 
$\End_{\ucA}(A) \setminus \bQ$. In that case, it follows from
\cite[(2.3)]{Oort88} that $\End_{\ucA}(E)$ is an imaginary quadratic
number field in which $p$ splits, if $E$ is ordinary. On the other 
hand, if $E$ is supersingular, then $\End_{\ucA}(E)$ contains no
such field. It follows that the largest common subfield to all
rings $\End_{\ucA}(E)$ is $\bQ$.

(ii) Assume that $\ucU$ is $K$-linear for some field $K$.
Then $K$ is a subfield of $\End_{\ucU}(\bG_a)$, and hence 
$\char(K) = p$. More generally, for any $n \geq 1$, we have 
a ring homomorphism $K \to \End_{\ucU}(W_n) = R/RV^n$
with the notation of \S \ref{subsubsec:ugpc}. As these homomorphisms 
are compatible with the natural maps 
$R/RV^{n+1} = \End_{\ucU}(W_{n+1}) \to \End_{\ucU}(W_n) = R/RV^n$,
we obtain a ring homomorphism $K \to \varprojlim R/RV^n$.
Since the right-hand side has characteristic $0$, this yields
a contradiction.

(iii) By Proposition \ref{prop:center}, it suffices to show that
the center of $R$ is $\bQ$. Moreover, every central element of $R$
is of the form
\[ z = \begin{pmatrix}
\sum_{A \in J} y_A & 0 \\
0 & (\sum_{T \in I} x_T) + x \\
\end{pmatrix}, \]
where each $y_A$ is central in $D_A$, each $x_T$ is central in $D_T$,
and $x \in k$; also, $x m_A = m_A y_A$ for all $m \in M_A$, and 
$x_T m_{T,A} = m_{T,A} y_A$ for all $m_{T,A} \in M_{T,A}$. 
Like in the proof of Proposition \ref{prop:veav}, 
it follows that $x \in \bQ$ and $y_A = x$ for all $A$. 
As a consequence, $x_T = y_A$ whenever $M_{T,A} \neq 0$.

To complete the proof, it suffices to show that for any $T \in I$,
there exists $A \in J$ such that $M_{T,A} \neq 0$; equivalently,
the $\Gamma$-module $A(\bar{k})_{\bQ}$ contains $\bfX(T)_{\bQ}$. 
We may choose an elliptic curve $E$ such that $E(k)$ is not torsion,
i.e., $E(k)_{\bQ} \neq 0$; then the $\Gamma$-module
$E(\bfX(T))(\bar{k})$ contains $\bfX(T)$. Thus, the desired assertion
holds for some simple factor $A$ of $E(\bfX(T))$.
\end{proof}

\begin{remark}\label{rem:qlin}
The above statement (iii) does not extend to the case where 
$\char(k) = p > 0$, since we then have $\ucC \cong \ucS \times \ucU$.
One may then show that the center of $\ucU$ contains the ring
of $p$-adic integers, $\bZ_p = W(\bF_p)$, with equality if and only 
if $k$ is infinite. In view of Remark \ref{rem:sav}, it follows that
the center of $\ucC$ contains $\bQ \times \bZ_p$, with equality if
$k$ is separably closed.
\end{remark}

\begin{proposition}\label{prop:savf}
The abelian category $\ucS$ is $\bQ$-linear and Hom-finite. It is 
Ext-finite if and only if $k$ satisfies the following condition:

\medskip

\noindent
{\rm (MW)} The vector space $A(k)_{\bQ}$ is finite-dimensional for any 
abelian variety $A$. 
\end{proposition}

\begin{proof}
Recall that every semi-abelian variety is divisible. In view of
(\ref{eqn:divisible}), it follows that $\ucS$ is $\bQ$-linear. 
It is Hom-finite in view of the Hom-finiteness of $\ucT$ and $\ucA$, 
combined with Proposition \ref{prop:homext}.

By that proposition, $\ucS$ is Ext-finite if and only if
the $\bQ$-vector space $\Ext^1_{\ucS}(A,T)$ is finite-dimensional  
for any abelian variety $A$ and any torus $T$. In view of
the isomorphism (\ref{eqn:wb}) and of the anti-equivalence of
categories (\ref{eqn:qcartier}), this amounts to the condition
that the vector space $\Hom^{\Gamma}(M,\wA(k_s)_{\bQ})$ be 
finite-dimensional for any $M \in \bQ\Gamma$-$\mod$. The latter 
condition is equivalent to (MW) by Lemma \ref{lem:twist}.
\end{proof}

\begin{remarks}\label{rem:mw}
(i) The above condition (MW) is a weak version of the Mordell-Weil
theorem, which asserts that the abelian group $A(k)$ is finitely 
generated for any abelian variety $A$ over a number field $k$.

\medskip

\noindent
(ii) The condition (MW) holds trivially if $k$ is locally finite,
as the abelian group $A(k)$ is torsion under that assumption.

\medskip

\noindent
(iii) Let $K/k$ be a finitely generated regular extension of fields
(recall that the regularity assumption means that $k$ is algebraically
closed in $K$, and $K$ is separable over $k$).
If (MW) holds for $k$, then it also holds for $K$ in view of
the Lang-N\'eron theorem (see \cite{Conrad-II} for a modern
proof of this classical result). As a consequence, (MW) holds 
whenever $k$ is finitely generated over a number field or 
a locally finite field.

One can also show that (MW) is invariant under purely transcendental
extensions (not necessarily finitely generated), by using the fact 
that every rational map from a projective space to an abelian variety 
is constant.
\end{remarks}

\begin{proposition}\label{prop:veavf}
Assume that $\char(k) = 0$. Then the $\bQ$-linear 
category $\ucV$ (resp.~$\ucC$) is Hom-finite if and only if 
$k$ is a number field. Under that assumption, $\ucV$ and $\ucC$ 
are Ext-finite as well.
\end{proposition}

\begin{proof}
If $\ucV$ is Hom-finite, then $\End_{\ucV}(\bG_a)$ is finite-dimensional
as a $\bQ$-vector space. Since $\End_{\ucV}(\bG_a) = \End_{\cU}(\bG_a) = k$, 
this means that $k$ is a number field. 

Conversely, if $k$ is a number field, then $\cU$ is 
Hom-finite, and hence so is $\ucL$. In view of Proposition
\ref{prop:homext}, it follows that $\ucC$ is Hom-finite,
and hence so is $\ucV$. By that proposition again, to prove that
$\ucC$ is Ext-finite, it suffices to check that the 
$\bQ$-vector space $\Ext^1_{\ucC}(A,L)$ is finite-dimensional
for any abelian variety $A$ and any connected linear algebraic
group $L$. Since $L = U \times T$ for a unipotent group $U$
and a torus $T$, this finiteness assertion follows by combining 
the isomorphisms (\ref{eqn:veav}), (\ref{eqn:wb}) and 
the Mordell-Weil theorem.
\end{proof}

\subsection{Finiteness representation type: an example}
\label{subsec:frt}

As in \S \ref{subsubsec:veav}, we consider the abelian category
$\ucV$ of vector extensions of abelian varieties over a field 
$k$ of characteristic $0$. Recall that $\ucV$ is $\bQ$-linear
and hereditary, and has enough projectives; its simple objects
are the additive group $\bG_a$ and the simple abelian varieties. 
In particular, $\ucV$ has infinitely many isomorphism classes
of simple objects. Also, by Proposition \ref{prop:veavf}, 
$\ucV$ is Hom-finite if and only if $k$ is a number field; 
then $\ucV$ is Ext-finite as well.

We now assume that $k$ is a number field. Choose a finite
set $F = \{ A_1, \ldots, A_r \}$ of simple abelian varieties, 
pairwise non-isogenous. Denote by $\ucV_F$ the Serre subcategory
of $\ucV$ generated by $F$. More specifically, the objects of
$\ucV_F$ are the algebraic groups obtained as extensions
\[ 0 \longrightarrow m_0 \bG_a \longrightarrow G \longrightarrow 
\bigoplus_{i=1}^r m_i A_i \longrightarrow 0, \]
where $m_0,m_1,\ldots, m_r$ are non-negative integers. The morphisms 
of $\ucV_F$ are the homomorphisms of algebraic groups.

By Theorem \ref{thm:equiv} and Remark \ref{rem:rmc}, we have 
an equivalence of categories
\[ \ucV_F \stackrel{\cong}{\longrightarrow} R_F\-\mod, \]
where $R_F$ denotes the triangular matrix ring
\[ \begin{pmatrix}
D_1 \oplus \cdots \oplus D_r & 0 \\
M_1 \oplus \cdots \oplus M_r & k \\
\end{pmatrix}. \]
Here $D_i := \End_{\ucV}(A_i)^{\op}$ is a division ring of finite
dimension as a $\bQ$-vector space, and $M_i := H^1(A_i,\cO_{A_i})^*$ 
is a $k$-$D_i$-bimodule, of finite dimension as a $k$-vector space.
Thus, $R_F$ is a finite-dimensional $\bQ$-algebra. Also, $R_F$ is 
hereditary, since so is $\ucV_F$. 

The $\bQ$-\emph{species} of $R_F$ is the directed graph
$\Gamma_F$ with vertices $0,1, \ldots,r$ and edges 
$\varepsilon_i := (0,i)$ for $i = 1,\ldots,r$.
The vertex $0$ is labeled with the field $k$, and each vertex
$i = 1,\ldots,r$ is labeled with the division ring $D_i$; each 
edge $\varepsilon_i$ is labeled with the $k$-$D_i$-bimodule $M_i$.
The category $R_F$-mod is equivalent to that of representations
of the $\bQ$-species $\Gamma_F$, as defined in \cite{DR} 
(see also \cite{Lemay}). 

The \emph{valued graph} of $\Gamma_F$ is the underlying non-directed
graph $\Delta_F$, where each edge $\{ 0 , i \}$ is labeled with the pair
$(\dim_k(M_i), \dim_{D_i}(M_i))$. As all edges contain $0$, we say 
that $0$ is a \emph{central vertex}; in particular, $\Delta_F$ 
is connected.

Recall that an Artin algebra is said to be 
\emph{of finite representation type} if it has only finitely 
many isomorphism classes of indecomposable modules of finite length.
In view of the main result of \cite{DR}, $R_F$ is 
of finite representation type if and only if $\Delta_F$ is a Dynkin 
diagram. By inspecting such diagrams having a central vertex, 
this is equivalent to $\Delta_F$ being a subgraph (containing $0$ 
as a central vertex) of one of the following graphs:

\bigskip

\xymatrix{\bfB_3: & {1} & \ar@{-}[l] {0} 
& \ar@{-}[l]_{(1,2)} {2} & & 
\bfC_3: \qquad {1} & \ar@{-}[l] {0} 
& \ar@{-}[l]_{(2,1)} {2}}

\bigskip

\xymatrix{\bfD_4: \qquad {1} & \ar@{-}[l] {0} & \ar@{-}[l] {2} & \\
 & \ar@{-}[u] {3} &  \\} 

\bigskip

\xymatrix{\bfG_2:  & {0} & \ar@{-}[l]_{(3,1)} {1} 
& & &  \bfG_2^{\op}:  & {0} & \ar@{-}[l]_{(1,3)} {1}}

\bigskip

\noindent
The subgraphs obtained in this way are as follows:

\bigskip

\xymatrix{\bfA_2:  & {0} & \ar@{-}[l] {1} 
& & & \bfB_2:  & {0} & \ar@{-}[l]_{(1,2)} {1}}
 
\bigskip

\xymatrix{\bfC_2: & {0} & \ar@{-}[l]_{(2,1)} {1}
& & & \bfA_3: & {1} & \ar@{-}[l] {0} & \ar@{-}[l] {2}}

\bigskip

\noindent
Here all unmarked edges have value $(1,1)$. We set 
\[ g_i := \dim(A_i) = \dim_k(M_i), \quad 
n_i := [D_i : \bQ], \quad n := [k : \bQ]. \] 
Then the label of each edge $\{ 0, i \}$ is $(g_i, \frac{g_i n}{n_i})$.
The above list entails restrictions on these labels, and hence on the simple
abelian varieties $A_i$ and the associated division rings $D_i$. We will 
work out the consequences of these restrictions in the case where $k$ 
is the field of rational numbers, which yields an especially simple result:

\begin{proposition}\label{prop:frt}
When $k = \bQ$, the algebra $R_F$ is of finite representation type
if and only if $\Delta_F = \bfD_4, \bfC_3, \bfG_2$ or a subgraph containing
$0$ as a central vertex (i.e., $\bfA_2, \bfA_3,\bfC_2$), and the abelian 
varieties $A_i$ satisfy the following conditions:

\begin{enumerate}

\item[{\rm $\bfD_4$:}] $A_1,A_2,A_3$ are elliptic curves.

\item[{\rm $\bfC_3$:}] $A_1$ is an elliptic curve and $A_2$ is a simple 
abelian surface with $[D_2 : \bQ] = 2$.

\item[{\rm $\bfG_2$:}] $A_1$ is a simple abelian threefold with
$[D_1 : \bQ] = 3$.

\end{enumerate}

\end{proposition}

\begin{proof}
Since $\Delta_F$ is a Dynkin diagram, we have $\dim_{\bQ}(M_i) = 1$
or $\dim_{D_i}(M_i) = 1$. In the former case, we have $g_i = 1$,
that is, $A_i$ is an elliptic curve. Moreover, 
$D_i \hookrightarrow \End_{\bQ}(M_i)$ as $M_i$ is a $k$-$D_i$-bimodule;
thus, $D_i = \bQ$. In the latter case, we have $g_i = n_i$. The result
follows from these observations via a case-by-case checking. 
\end{proof}

\begin{remarks}\label{rem:frt}
(i) We may view the Dynkin diagrams $\bfD_4$ and $\bfC_3$ as unfoldings 
of $\bfG_2$. In fact, a similar picture holds for the abelian varieties
under consideration: let $A := A_1 \oplus \cdots \oplus A_r$, then
$A_1,\ldots,A_r$ satisfy the assertion of Proposition \ref{prop:frt}
if and only if $\dim(A) = 3 = \dim \End_{\ucV}(A)_{\bQ}$. 
In the ``general'' case where $A$ is simple, this yields type $\bfG_2$; 
it ``specializes'' to types $\bf C_3$ and $\bfD_4$.

\medskip

\noindent
(ii) When $R_F$ is of finite representation type, its indecomposable
modules of finite length are described by the main result of \cite{DR}:
the isomorphism classes of such modules correspond bijectively to 
the positive roots of the root system with Dynkin diagram $\Delta_F$, 
by assigning with each module its dimension type (the sequence of 
multiplicities of the simple modules). 
This yields a case-by-case construction of the indecomposable objects 
of $\ucV_F$. For example, in type $\bfD_4$, the indecomposable object 
associated with the highest root (i.e., with the sequence 
of multiplicities $m_0 = 2$, $m_1 = m_2 = m_3 = 1$) is the quotient 
of the universal vector extension $E(A_1 \oplus A_2 \oplus A_3)$ 
by a copy of $\bG_a$ embedded diagonally in 
$U(A_1 \oplus A_2 \oplus A_3) \cong 3 \bG_a$.
But we do not know any uniform construction of indecomposable
objects for all types, along the lines of (i).

\medskip

\noindent
(iii) All the abelian varieties $A_i$ over $\bQ$ that occur in 
Proposition \ref{prop:frt} satisfy the condition that $\End(A_i)_{\bQ}$
is a field of dimension equal to $\dim(A_i)$. This condition defines 
the class of \emph{abelian varieties of $\GL_2$-type}, introduced by
Ribet in \cite{Ribet}; it includes all elliptic curves over 
$\bQ$, and also the abelian varieties associated with certain modular
forms via a construction of Shimura (see \cite[Thm.~7.14]{Shimura}). 
Assuming a conjecture of Serre on Galois representations, Ribet 
showed in \cite{Ribet} that this construction yields all abelian 
varieties of $\GL_2$-type up to isogeny. 

Examples of abelian varieties of $\GL_2$-type have been obtained by  
Gonz\'alez, Gu\'ardia and Rotger in dimension $2$ (see 
\cite[Cor.~3.10]{GGR}), and by Baran in dimension $3$ (see 
\cite{Baran}). 

\medskip

\noindent
(iv) Still assuming that $k$ is a number field, the question of 
characterizing finite representation type makes sense,
more generally, for the Serre subcategory $\ucC_{E,F} \subset \ucC$ 
generated by a finite set $E$ of simple linear algebraic groups and 
a finite set $F$ of simple abelian varieties, pairwise non-isogenous 
(so that $\ucC_{\bG_a,F} = \ucV_F$). The abelian category
$\ucC_{E,F}$ is equivalent to $R_{E,F}$-mod, where $R_{E,F}$ is a 
triangular matrix algebra of finite dimension over $\bQ$, constructed
as above. The $\bQ$-species associated with $R_{E,F}$ is the directed
graph $\Gamma_{E,F}$ with vertices $E \sqcup F$ and edges $(i,j)$ 
for all $i \in E$, $j \in F$ such that 
$\Ext^1_{\ucC}(A_j,L_i) \neq 0$; here $A_j$ (resp.~$L_i$) denotes the
corresponding simple abelian variety (resp.~linear algebraic group).
In particular, if $\bG_a \in E$ then the associated vertex is linked
to all vertices in $F$, but some simple tori need not be linked 
to some simple abelian varieties. Each vertex $v$ is labeled with
the division ring $D_v$ opposite to the endomorphism ring of the 
corresponding simple module, and each edge $(i,j)$ is labeled
with the $D_i$-$D_j$-bimodule $\Ext^1_{\ucC}(A_j,L_i)$. Then again,
the category $R_{E,F}$ is equivalent to that of representations of
the $\bQ$-species $\Gamma_{E,F}$; it is of finite representation
type if and only if each connected component of the associated
valued graph $\Delta_{E,F}$ is a Dynkin diagram. Note that such
a diagram comes with a bipartition (by vertices in $E$, $F$).

To obtain a full characterization of finite representation type
in this generality, we would need detailed information on the
structure of $\Gamma$-module of $A(\bar{k})_{\bQ}$ for any $A \in F$.
But it seems that very little is known on this topic. For example,
just take $E := \{ \bG_m \}$; recall that 
$\Ext^1_{\ucC}(A,\bG_m) \cong A(k)_{\bQ}$ for any abelian variety $A$.
We may thus assume that the finitely generated abelian group $A(k)$ 
is infinite for any $A \in F$; then $\bG_m$ is a central vertex
of $\Gamma_{\bG_m,F}$. Arguing as in the proof of Proposition
\ref{prop:frt}, one obtains a similar characterization of finite
representation type in terms of Dynkin diagrams satisfying the
following conditions:

\begin{enumerate}

\item[{\rm $\bfD_4$:}] $[D_i : \bQ] = 3 = \dim A_i(k)_{\bQ}$
for $i = 1,2,3$.

\item[{\rm $\bfC_3$:}] $[D_1 : \bQ] = 1 = \dim A_1(k)_{\bQ}$
and $[D_2 : \bQ] = 2 = \dim A_2(k)_{\bQ}$.

\item[{\rm $\bfG_2$:}] $[D_1 : \bQ] = 3 = \dim A_1(k)_{\bQ}$.

\end{enumerate}

As a consequence, all simple abelian varieties $A$ occuring in
$F$ must satisfy $[D_A : \bQ ] \leq 3$ (in particular, $D_A$
is commutative) and $\dim A(k)_{\bQ} = [D_A : \bQ ]$. We do not
know whether such abelian varieties exist in arbitrary large
dimensions.
\end{remarks}

\medskip

\noindent
{\bf Acknowledgements}.
I had the opportunity to present the results from \cite{Brion-II} 
and some results from the present paper, at the Lens 2016 mini-courses 
and the International Conference on Representations of Algebras, Syracuse, 
2016. I thank the organizers of both events for their invitation, and 
the participants for stimulating questions. Also, I warmly thank 
Claire Amiot, Brian Conrad, St\'ephane Guillermou, Henning Krause, 
George Modoi, Idun Reiten and Ga\"el R\'emond for very 
helpful discussions or e-mail exchanges.

\bibliographystyle{amsalpha}

\end{document}